\documentclass[12pt]{amsart}
\DeclareMathOperator{\comm}{comm}
\DeclareMathOperator{\free}{free}
\DeclareMathOperator{\even}{even}
\DeclareMathOperator{\aut}{aut}
\DeclareMathOperator{\ev}{ev}
\DeclareMathOperator{\id}{id}
\usepackage{amssymb, xcolor}
\definecolor{lgray}{gray}{0.7}     
\renewcommand{\phi}{\varphi}

\newsavebox{\boxpaarpart}
   \savebox{\boxpaarpart}
   { \begin{picture}(0.7,0.5)
     \put(0,0){\line(0,1){0.4}}
     \put(0.5,0){\line(0,1){0.4}}
     \put(0,0.4){\line(1,0){0.5}}
     \end{picture}}
\newsavebox{\boxbaarpart}
   \savebox{\boxbaarpart}
   { \begin{picture}(0.7,0.5)
     \put(0,0){\line(0,1){0.4}}
     \put(0.5,0){\line(0,1){0.4}}
     \put(0,0){\line(1,0){0.5}}
     \end{picture}}
\newsavebox{\boxpaarbaarpart}
   \savebox{\boxpaarbaarpart}
   { \begin{picture}(0.7,1.2)
     \put(0,0){\line(0,1){0.4}}
     \put(0.5,0){\line(0,1){0.4}}
     \put(0,0.4){\line(1,0){0.5}}
     \put(0,0.6){\line(0,1){0.4}}
     \put(0.5,0.6){\line(0,1){0.4}}
     \put(0,0.6){\line(1,0){0.5}}
     \end{picture}}
\newsavebox{\boxpaarbaareps}
   \savebox{\boxpaarbaareps}
   { \begin{picture}(0.7,1.4)
     \put(0,0){\line(0,1){0.2}}
     \put(0.5,0){\line(0,1){0.2}}
     \put(0,0.2){\line(1,0){0.5}}
     \put(0,1){\line(0,1){0.2}}
     \put(0.5,1){\line(0,1){0.2}}
     \put(0,1){\line(1,0){0.5}}
     \put(0.1,0.2){\line(0,1){0.15}}
     \put(0.1,0.5){\line(0,1){0.2}}
     \put(0.1,0.85){\line(0,1){0.15}}          
     \put(0.3,0.4){\footnotesize{$\varepsilon$}}
     \end{picture}}
\newsavebox{\boxcrosspart}
   \savebox{\boxcrosspart}
   { \begin{picture}(0.5,0.8)
     \put(0,-0.2){\line(1,2){0.4}}
     \put(0.4,-0.2){\line(-1,2){0.4}}
     \end{picture}}
\newsavebox{\boxidpart}
   \savebox{\boxidpart}
   { \begin{picture}(0.3,1)
     \put(0.1,0){\line(0,1){0.8}}
     \end{picture}}
\newsavebox{\boxididpart}
   \savebox{\boxididpart}
   { \begin{picture}(0.7,1)
     \put(0,0){\line(0,1){0.8}}
     \put(0.5,0){\line(0,1){0.8}}
     \end{picture}}
\newsavebox{\boxvierpartrot}
   \savebox{\boxvierpartrot}
   { \begin{picture}(0.5,0.8)
     \put(0,-0.2){\line(0,1){0.3}}
     \put(0.4,-0.2){\line(0,1){0.3}}
     \put(0,0.1){\line(1,0){0.4}}
     \put(0,0.4){\line(0,1){0.3}}
     \put(0.4,0.4){\line(0,1){0.3}}
     \put(0,0.4){\line(1,0){0.4}}
     \put(0.2,0.1){\line(0,1){0.3}}
     \end{picture}}
\newsavebox{\boxdreipartrot}
   \savebox{\boxdreipartrot}
   { \begin{picture}(0.5,0.8)
     \put(0,0.4){\line(0,1){0.3}}
     \put(0.4,0.4){\line(0,1){0.3}}
     \put(0,0.4){\line(1,0){0.4}}
     \put(0.2,-0.2){\line(0,1){0.6}}
     \end{picture}}
\newsavebox{\boxCross}
   \savebox{\boxCross}
   { \begin{picture}(0.6,1)
     \put(0,0){\line(2,3){0.6}}
     \put(0.6,0){\line(-2,3){0.6}}
     \put(0.05,0){\line(2,3){0.6}}
     \put(0.65,0){\line(-2,3){0.6}}
     \end{picture}}  
\newsavebox{\boxWideCross}
   \savebox{\boxWideCross}
   { \begin{picture}(1,1)
     \put(0,0){\line(1,1){1}}
     \put(1,0){\line(-1,1){1}}
     \put(0.05,0){\line(1,1){1}}
     \put(1.05,0){\line(-1,1){1}}
     \end{picture}}       
\newsavebox{\boxId}
   \savebox{\boxId}
   { \begin{picture}(0.6,1)
     \put(-0.05,0){{\color{lgray}\line(0,1){0.9}}}
     \put(0.9,0){{\color{lgray}\line(0,1){0.9}}}
     \put(0,0){{\color{lgray}\line(0,1){0.9}}}
     \put(0.8,0){{\color{lgray}\line(0,1){0.9}}}
     \put(0.05,0){{\color{lgray}\line(0,1){0.9}}}
     \put(0.85,0){{\color{lgray}\line(0,1){0.9}}}
     \end{picture}}   
\newsavebox{\boxBlackId}
   \savebox{\boxBlackId}
   { \begin{picture}(0.6,1)
     \put(-0.05,0){\line(0,1){0.9}}
     \put(0.9,0){\line(0,1){0.9}}
     \put(0,0){\line(0,1){0.9}}
     \put(0.8,0){\line(0,1){0.9}}
     \put(0.05,0){\line(0,1){0.9}}
     \put(0.85,0){\line(0,1){0.9}}
     \end{picture}}      
\newsavebox{\boxPaar}
   \savebox{\boxPaar}
   { \begin{picture}(0.6,1)
     \put(0,0){{\color{lgray}\line(1,0){0.7}}}
     \put(0,0.9){{\color{lgray}\line(1,0){0.7}}}
     \put(0.05,0.05){{\color{lgray}\line(1,0){0.6}}}
     \put(0.05,0.85){{\color{lgray}\line(1,0){0.6}}}
     \put(0.1,0.1){{\color{lgray}\line(1,0){0.5}}}
     \put(0.1,0.8){{\color{lgray}\line(1,0){0.5}}}
     \end{picture}}        
\newsavebox{\boxMixCrossId}
   \savebox{\boxMixCrossId}
   { \begin{picture}(1.2,1.2)
     \put(0,0.1){\usebox{\boxCross}}
     \put(-0.1,0.1){\usebox{\boxId}}     
     \put(0,0){\framebox(1.2,1.1)}
     \put(0.2,-0.2){\line(0,1){0.2}}
     \put(1,-0.2){\line(0,1){0.2}}
     \put(0.2,1.1){\line(0,1){0.2}}
     \put(1,1.1){\line(0,1){0.2}}
     \end{picture}}     
\newsavebox{\boxMixCrossPaar}
   \savebox{\boxMixCrossPaar}
   { \begin{picture}(1.2,1.2)
     \put(-0.2,0.05){\usebox{\boxWideCross}}
     \put(-0.05,0.1){\usebox{\boxPaar}}     
     \put(0,0){\framebox(1.2,1.1)}
     \put(0.2,-0.2){\line(0,1){0.2}}
     \put(1,-0.2){\line(0,1){0.2}}
     \put(0.2,1.1){\line(0,1){0.2}}
     \put(1,1.1){\line(0,1){0.2}}
     \end{picture}}   
\newsavebox{\boxMixIdPaar}
   \savebox{\boxMixIdPaar}
   { \begin{picture}(1.2,1.2)
     \put(-0.1,0.1){\usebox{\boxBlackId}}
     \put(-0.05,0.1){\usebox{\boxPaar}}      
     \put(0,0){\framebox(1.2,1.1)}
     \put(0.2,-0.2){\line(0,1){0.2}}
     \put(1,-0.2){\line(0,1){0.2}}
     \put(0.2,1.1){\line(0,1){0.2}}
     \put(1,1.1){\line(0,1){0.2}}
     \end{picture}}              
\newcommand{\paarpart}{\usebox{\boxpaarpart}}
\newcommand{\baarpart}{\usebox{\boxbaarpart}}
\newcommand{\paarbaarpart}{\usebox{\boxpaarbaarpart}}

\newcommand{\crosspart}{\usebox{\boxcrosspart}}

\newcommand{\ididpart}{\usebox{\boxididpart}}
\newcommand{\vierpartrot}{\usebox{\boxvierpartrot}}
\newcommand{\dreipartrot}{\usebox{\boxdreipartrot}}
\newcommand{\MixCrossId}{\usebox{\boxMixCrossId}}
\newcommand{\MixCrossPaar}{\usebox{\boxMixCrossPaar}}
\newcommand{\MixIdPaar}{\usebox{\boxMixIdPaar}}

\usepackage{calc}
\newcounter{PartitionDepth}
\newcounter{PartitionLength}
\newcommand{\parti}[2]{
 \begin{picture}(#2,#1)
 \setcounter{PartitionDepth}{-1-#1}
 \put(#2,\thePartitionDepth){\line(0,1){#1}}
 \end{picture}}
\newcommand{\partii}[3]{
 \begin{picture}(#3,#1)
 \setcounter{PartitionLength}{#3-#2}
 \setcounter{PartitionDepth}{-1-#1}
 \put(#2,\thePartitionDepth){\line(0,1){#1}}     
 \put(#3,\thePartitionDepth){\line(0,1){#1}}
 \put(#2,\thePartitionDepth){\line(1,0){\thePartitionLength}}
 \end{picture}}
\newcommand{\partiii}[4]{
 \begin{picture}(#4,#1)
 \setcounter{PartitionLength}{#4-#2}
 \setcounter{PartitionDepth}{-1-#1}
 \put(#2,\thePartitionDepth){\line(0,1){#1}}
 \put(#3,\thePartitionDepth){\line(0,1){#1}}
 \put(#4,\thePartitionDepth){\line(0,1){#1}}
 \put(#2,\thePartitionDepth){\line(1,0){\thePartitionLength}} 
 \end{picture}}

\newcommand{\upparti}[2]{
 \begin{picture}(#2,#1)
 \setcounter{PartitionDepth}{#1}
 \put(#2,0){\line(0,1){#1}}
 \end{picture}}
\newcommand{\uppartii}[3]{
 \begin{picture}(#3,#1)
 \setcounter{PartitionLength}{#3-#2}
 \setcounter{PartitionDepth}{#1}
 \put(#2,0){\line(0,1){#1}}     
 \put(#3,0){\line(0,1){#1}}
 \put(#2,\thePartitionDepth){\line(1,0){\thePartitionLength}}
 \end{picture}}
\newcommand{\uppartiii}[4]{
 \begin{picture}(#4,#1)
 \setcounter{PartitionLength}{#4-#2}
 \setcounter{PartitionDepth}{#1}
 \put(#2,0){\line(0,1){#1}}
 \put(#3,0){\line(0,1){#1}}
 \put(#4,0){\line(0,1){#1}}
 \put(#2,\thePartitionDepth){\line(1,0){\thePartitionLength}} 
 \end{picture}}

\begin{document}

\newtheorem{theorem}{Theorem}[section]
\newtheorem{lemma}[theorem]{Lemma}
\newtheorem{proposition}[theorem]{Proposition}
\newtheorem{corollary}[theorem]{Corollary}

\theoremstyle{definition}
\newtheorem{definition}[theorem]{Definition}
\newtheorem{example}[theorem]{Example}
\newtheorem{xca}[theorem]{Exercise}
\newtheorem{notation}[theorem]{Notation}

\theoremstyle{remark}
\newtheorem{remark}[theorem]{Remark}


\newcommand{\NN}{ {\mathbb N} }
\newcommand{\ZZ}{ {\mathbb Z} }
\newcommand{\RR}{ {\mathbb R} }
\newcommand{\CC}{{\mathbb C}}
\newcommand{\cD}{{\mathcal D}}
\newcommand{\cP}{ {\mathcal P} }
\newcommand{\cPP}{{\mathcal{PS}}}
\newcommand{\cF}{{\mathcal F}}
\newcommand{\HH}{{\mathcal H}}
\newcommand{\GG}{{\mathcal G}}
\newcommand{\cV}{{\mathcal V}}
\newcommand{\cW}{{\mathcal W}}
\newcommand{\cM}{{\mathcal M}}
\newcommand{\KK}{{\kappa}}
\newcommand{\cU}{{\mathcal U}}
\newcommand{\UU}{\mathcal{U}}

\newcommand{\cX}{{\mathcal X}}
\newcommand{\dd}{\mathbf{d}}
\newcommand{\cI}{{T}}
\newcommand{\ff}{\varphi}
\newcommand{\la}{\langle}
\newcommand{\ra}{\rangle}
\newcommand{\tr}{\mathrm{tr}}
\newcommand{\Tr}{\mathrm{Tr}}
\newcommand{\EE}{\mathrm{E}}
\newcommand{\ct}{{\bf t}}
\newcommand{\cB}{\mathcal{B}}
\newcommand{\cA}{\mathcal{A}}
\newcommand{\kt}{\kk^\ct}
\newcommand{\odo}{\otimes\cdots\otimes}
\newcommand{\cyc}{\mathbf{c}\,}
\newcommand{\SNC}{S_{NC}}
\newcommand{\SNCc}{\overline{\SNC}}
\newcommand{\SNCe}{S^{(\epsilon)}_{NC}}
\newcommand{\NC}{NC}
\newcommand{\cyclic}{{\text{\rm cyc}}}
\newcommand{\lin}{{\text{\rm lin}}}
\newcommand{\Wg}{\mathrm{Wg}}
\newcommand{\fluct}{\zeta}
\newcommand{\Nb}{{(N)}}
\newcommand{\cS}{\mathcal{S}}
\newcommand{\ii}{\text{\bf i}}
\newcommand{\jj}{\text{\bf j}}
\newcommand{\kkk}{\text{\bf k}}
\newcommand{\rr}{\text{\bf r}}

\hyphenation{Voi-cu-les-cu}


\newcommand{\cc}{\kappa}
\newcommand{\cG}{\mathcal{G}}
\newcommand{\cH}{\mathcal{H}}
\newcommand{\cR}{\mathcal{R}}
\newcommand{\cC}{\mathcal{C}}
\newcommand{\cov}{\mathrm{cov}}
\newcommand{\kk}{\kappa}
\newcommand{\moeb}{\mathrm{\text{M\"ob}}}
\newcommand{\cWg}{C}
\newcommand{\tovert}{\to}
\newcommand{\IZ}{\mathrm{IZ}}
\newcommand{\bF}{\mathbb{F}}
\newcommand{\Distrhigher}{D_{higher}}
\newcommand{\ab}{\allowbreak}
\newcommand{\ol}{\overline}
\newcommand{\ie}{\textit{i.e.}\,}
\newcommand{\ee}{\varepsilon}
\newcommand{\HHc}{\HH^{\circ}}

\newcommand{\ky}{\kappa}

\newcommand{\GUE}{\text{GUE}}

\title{Quantum groups with partial commutation relations}
\author{Roland Speicher}
\author{Moritz Weber}
\address{Saarland University, Fachbereich Mathematik, Postfach 151150, 66041 Saarbr\"ucken, Germany}
\email{speicher@math.uni-sb.de, weber@math.uni-sb.de}
\date{\today}
\thanks{Both authors were partially funded by the ERC Advanced Grant on Non-Commutative Distributions in Free Probability, held by Roland Speicher.}

\begin{abstract}
We define new noncommutative spheres with partial commutation relations for the coordinates. We investigate the quantum groups acting maximally on them, which yields new quantum versions of the orthogonal group: They are partially commutative in a way such that they do \emph{not} interpolate between the classical and the free quantum versions of the orthogonal group. Likewise we define non-interpolating, partially commutative quantum versions of the symmetric group recovering Bichon's quantum automorphism groups of graphs. They fit with the mixture of classical and free independence as recently defined by Speicher and Wysoczanski (rediscovering $\Lambda$-freeness of Mlotkowski), due to some weakened version of a de Finetti theorem. 
\end{abstract}

\maketitle

\section{Introduction}

Motivated by the recent preprint on mixtures of classical and free indepence by Wysoczanski and the first author \cite{SpW} -- where the notion of and results on $\Lambda$-freeness of Mlotkowski \cite{Mlo} were rediscovered -- we ask for the corresponding quantum symmetries. The mixture of independences goes as follows. Let $\ee=(\ee_{ij})_{i,j\in \{1,\ldots,n\}}$ be a symmetric matrix with $\ee_{ij}\in\{0,1\}$ and $\ee_{ii}=0$. If variables $x_1,\ldots,x_n$ are $\ee$-independent, then:
\begin{itemize}
\item $x_i$ and $x_j$ are free in the case $\ee_{ij}=0$
\item and $x_i$ and $x_j$ are independent in the case $\ee_{ij}=1$ (in particular $x_ix_j=x_jx_i$ in this situation).
\end{itemize}
If all entries of $\ee$ are zero ($\ee=\ee_{\free}$), we obtain Voiculescu's free independence; if all non-diagonal entries of $\ee$ are one ($\ee=\ee_{\comm}$), we obtain classical independence. 

It is well-known that independences can be characterized via distributional symmetries with the help of de Finetti type theorems. For instance, classical independence is equivalent to invariance under actions of the symmetric group $S_n$, whereas free independence is characterized \cite{SpK} by invariance under actions of Wang's \cite{WangSym} quantum symmetric group $S_n^+$. Our aim is to find the right quantum groups corresponding to the above $\ee$-independence.

For doing so, we first define noncommutative $\ee$-spheres $S_{\RR,\ee}^{n-1}$ by:
\[C(S_{\RR,\ee}^{n-1}):=C^*(x_1,\ldots,x_n\;|\; x_i=x_i^*, \sum_i x_i^2=1, x_ix_j=x_jx_i, \textnormal{ if } \ee_{ij}=1)\]
For $\ee=\ee_{\comm}$ the above $C^*$-algebra is commutative and it is nothing but the algebra of continuous functions over the real sphere $S_{\RR}^{n-1}\subset\RR^n$. For $\ee=\ee_{\free}$ we obtain Banica and Goswami's \cite{BanGos} free version $S^{n-1}_{\RR,+}$ of the sphere. In both cases we know the (quantum) group acting maximally on the sphere. In the first case, this is the orthogonal group $O_n\subset M_n(\RR)$ whereas in the second case \cite{BanGos}, it is Wang's \cite{WangOrth} free orthogonal quantum group $O_n^+$. We  define an $\ee$-version $O_n^\ee$ of these objects by
\[C(O_n^\ee)=C^*(u_{ij}, i,j=1,\ldots,n\;|\; u_{ij}=u_{ij}^*, u \textnormal{ is orthogonal}, R^{\ee}\textnormal{ hold}),\]
where $R^\ee$ are the relations:
\[u_{ik}u_{jl}=\begin{cases}
u_{jl}u_{ik}&\text{if  $\ee_{ij}=1$ and $\ee_{kl}=1$}\\
u_{jk}u_{il}&\text{if  $\ee_{ij}=1$ and $\ee_{kl}=0$}\\
u_{il}u_{jk}&\text{if  $\ee_{ij}=0$ and $\ee_{kl}=1$}
\end{cases}\]
Moreover, we define an $\ee$-version $S_n^\ee\subset O_n^\ee$ of the symmetric group by
\[C(S_n^\ee)= C^*(u_{ij}\;|\; u_{ij}=u_{ij}^*=u_{ij}^2,\sum_ku_{ik}=\sum_ku_{kj}=1\;\forall i,j, \mathring R^\ee \textnormal{ hold}),\]
where $\mathring R^\ee$ are the relations:
\[u_{ik}u_{jl}=\begin{cases}
u_{jl}u_{ik}&\text{if  $\ee_{ij}=1$ and $\ee_{kl}=1$}\\
0&\text{if  $\ee_{ij}=1$ and $\ee_{kl}=0$}\\
0&\text{if  $\ee_{ij}=0$ and $\ee_{kl}=1$}
\end{cases}\]
Note that the quotient of $C(S_n^+)$ by $R^\ee$ coincides with the one by $\mathring R^\ee$ (Lemma \ref{LemREpsAeqRRing}). Again, the choice  $\ee=\ee_{\comm}$ yields $O_n^\ee=O_n$ and $S_n^\ee=S_n$, whereas $\ee=\ee_{\free}$ yields $O_n^\ee=O_n^+$ and $S_n^\ee=S_n^+$. The quantum groups $S_n^\ee$ coincide with Bichon's quantum automorphism groups of graphs \cite{BicQAG} and they are quantum subgroups of Banica's quantum automorphisms of graphs \cite{BanQAG}. The latter one are given by quotients of $C(S_n^+)$ by the relations $u\ee=\ee u$.

On the level of groups (or monoids) such mixed commutation relations have been studied extensively under names such as ``right angled Artin groups'', ``free partially commutative groups'', ``trace groups'', ``graph groups'', ``Cartier-Foata monoids'', ``trace monoids'' etc, see for instance \cite{Cha,Foata} or the references in \cite{SpW}. It is also linked to the following Coxeter groups (see Def \ref{DefCox}):
\[\langle a_1,\ldots,a_n\;|\; (a_ia_j)^{m_{ij}}=e\rangle,\qquad 
m_{ij}=\begin{cases} 1 &\textnormal{ if } i=j\\2 &\textnormal{ if } \ee_{ij}=1\\\infty &\textnormal{ if }\ee_{ij}=0\end{cases}\]

\section{Main results}

Our main result about the $\ee$-sphere $S_{\RR,\ee}^{n-1}$ (or rather about its associated $C^*$-algebra $C(S_{\RR,\ee}^{n-1})$) is the following.

\begin{theorem}[Thm. \ref{VglSphaeren}]
The $\ee$-spheres are noncommutative as soon as $\ee\neq\ee_{\comm}$ and we have $S_{\RR,\ee}^{n-1}\neq S_{\RR,\ee'}^{n-1}$ for $\ee\neq\ee'$. In particular,
\[S_{\RR}^{n-1}\subsetneq S_{\RR,\ee}^{n-1}\subsetneq S_{\RR,+}^{n-1}\]
for $\ee\neq\ee_{\comm}$ and $\ee\neq\ee_{\free}$.
\end{theorem}

Here, $S_{\RR,\ee}^{n-1}\neq S_{\RR,\ee'}^{n-1}$ means that there is no isomorphism from $C(S_{\RR,\ee}^{n-1})$ to $C(S_{\RR,\ee'}^{n-1})$ mapping $x_i\mapsto x_i$.
The results on $O_n^\ee$ may be summarized as follows.

\begin{theorem}[Prop. \ref{OnEpsNoncommutative}, Thm. \ref{Action}]
The quantum groups $O_n^\ee$ are no groups as soon as $\ee\neq\ee_{\comm}$ and we have $O_n^\ee\neq O_n^{\ee'}$ for $\ee\neq\ee'$. The quantum group $O_n^\ee$ acts maximally on the $\ee$-sphere.
\end{theorem}

Note that while the $\ee$-sphere $S_{\RR,\ee}^{n-1}$ interpolates the commutative sphere $S_{\RR}^{n-1}$ and the free sphere $S_{\RR,+}^{n-1}$, this is \emph{not} the case for the $\ee$-orthogonal quantum groups  --
if $\ee\neq \ee_{\comm}$ and $\ee\neq \ee_{\free}$, we have (see Prop. \ref{PropNoInterpolation}):
\[O_n\not\subset O_n^\ee\subsetneq O_n^+\qquad\textnormal{and}\qquad S_n\not\subset O_n^\ee\subsetneq O_n^+\]

Our original question about the symmetries of $\ee$-independence is answered by the following weak version of a de Finetti theorem.

\begin{theorem}[Thm. \ref{deFinetti}]
Let $x_1,\ldots,x_n$ be selfadjoint random variables in a noncommutative probability space $(A,\phi)$ such that $x_ix_j=x_jx_i$ if $\ee_{ij}=1$.  If $x_1,\ldots,x_n$ are $\ee$-independent and identically distributed, then their distribution is invariant under $S_n^\ee$.
\end{theorem}

We extend this de Finetti Theorem to the invariances by the quotients of $H_n^+$, $B_n^+$ and $O_n^+$ by $\mathring R^\ee$.

\section{Preliminaries on $\ee$-independence}

The notion of $\ee$-independence as a mixture of classical and free independence has been introduced by the first author and Wysoczanski \cite{SpW} very recently, rediscovering Mlotkowski's $\Lambda$-independence \cite{Mlo}. We review its main features here.
Throughout the whole  article we denote by $\ee$ an $n\times n$-matrix such that:
\begin{itemize}
\item $\ee_{ij}\in\{0,1\}$ for all $i,j=1,\ldots,n$
\item $\ee$ is symmetric
\item $\ee_{ii}=0$ for all $i=1,\ldots,n$
\end{itemize}

\begin{definition}\label{DefEpsInd}
Let $(A,\phi)$ be a noncommutative probability space. We say that unital subalgebras $A_1,\ldots,A_n\subset A$ are \emph{$\ee$-independent}, if we have:
\begin{itemize}
\item[(i)] The algebras $A_i$ and $A_j$ commute, if $\ee_{ij}=1$.
\item[(ii)] Moreover, for any $k\in\NN$ and any choice $a_1,\ldots,a_k\in A$ with $a_j\in A_{i(j)}$ and the properties
\begin{itemize}
\item[$\bullet$] $\phi(a_j)=0$ for all $j=1,\ldots,k$,
\item[$\bullet$] and for any $1\leq p<r\leq k$ with $i(p)=i(r)$ there is a $q$ with $p<q<r$ such that $\ee_{i(p)i(q)}=0$ and $i(p)\neq i(q)$,
\end{itemize}
 we have $\phi(a_1\cdots a_k)=0$.
\end{itemize}
(Selfadjoint) variables $x_1,\ldots,x_n\in A$ are $\ee$-independent, if the algebras $\textnormal{alg}(x_j,1)\subset A$ are $\ee$-independent.
\end{definition}

\begin{example}\label{ExEps}
Here are a few examples of $\ee$-independent variables.
\begin{itemize}
\item[(a)] Let $\ee_{\comm}\in M_n(\{0,1\})$ be the matrix defined by  $\ee_{ij}=1$ for all $i\neq j$ and $\ee_{ii}=0$ for all $i$. Variables $x_1,\ldots, x_n$ are $\ee$-independent with respect to $\ee_{\comm}$ if and only if they all commute and are classically independent. Indeed, the constraint on the indices in Definition \ref{DefEpsInd}(ii) yields that all indices must be mutually different, in case $\ee=\ee_{\comm}$. Now, by the usual centering trick on $a_j:=x_{i(j)}^{m_j}-\phi(x_{i(j)}^{m_j})$ with mutually different indices $i(j)$, we infer
\[\phi(x_{i(1)}^{m_1}\cdots x_{i(k)}^{m_k})=\prod_{j=1}^k \phi(x_{i(j)}^{m_j})\]
and hence classical independence. See also \cite[Prop. 3.2]{SpW}.
\item[(b)] Let $\ee_{\free}\in M_n(\{0,1\})$ be defined by $\ee_{ij}=0$ for all $i,j$. Variables $x_1,\ldots, x_n$ are $\ee$-independent with respect to $\ee_{\free}$ if and only if they are freely independent. Note that if $\ee=\ee_{\free}$, the constraint on the indices in Definition \ref{DefEpsInd}(ii) yields that neighbouring indices must be different. See also \cite[Prop. 3.2]{SpW}.
\item[(c)] Let $\ee\in M_{n+m}(\{0,1\})$ be the matrix given by:
\[\ee=\begin{pmatrix}\ee_{\comm} &0\\0&0\end{pmatrix},\qquad \ee_{ij}=\begin{cases} 1 &\textnormal{ if } i\leq n\textnormal{ and } j\leq n\textnormal{ and } i\neq j\\0&\textnormal{ otherwise}\end{cases}\]
If variables $x_1,\ldots, x_n,x_{n+1},\ldots,x_{n+m}$ are $\ee$-independent with respect to this matrix, then
\begin{itemize}
\item[(i)] $x_1,\ldots,x_n$ are classically independent,
\item[(ii)] $x_{n+1},\ldots,x_{n+m}$ are freely independent,
\item[(iii)] and $\{x_1,\ldots,x_n\}$ and $\{x_{n+1},\ldots,x_{n+m}\}$ are free.
\end{itemize}
\item[(d)] The iterated grouping of variables 
\begin{itemize}
\item[(i)] $x_1$ and $x_2$ are independent,
\item[(ii)] $x_3$ and $x_4$ are independent,
\item[(iii)] and $\{x_1,x_2\}$ is free from $\{x_3,x_4\}$
\end{itemize}
is represented by the following matrix  $\ee\in M_4(\{0,1\})$:
\[\ee=\begin{pmatrix}0&1&0&0\\1&0&0&0\\0&0&0&1\\0&0&1&0\end{pmatrix}\]
\item[(e)] Swapping the terms ``free'' and ``independent'' in (d), we obtain:
\[\ee=\begin{pmatrix}0&0&1&1\\0&0&1&1\\1&1&0&0\\1&1&0&0\end{pmatrix}\]
\item[(f)] A constellation that cannot be obtained from iterated grouping is the following (the motivating example in \cite{SpW}):
\begin{itemize}
\item[(i)] $x_i$ and $x_{i+1}$ are free, for $i=1,2,3,4$, as well as $x_5$ and $x_1$,
\item[(ii)] but all other pairs are independent.
\end{itemize}
The matrix  $\ee\in M_5(\{0,1\})$ in this situation is:
\[\ee=\begin{pmatrix}0&0&1&1&0\\0&0&0&1&1\\1&0&0&0&1\\1&1&0&0&0\\0&1&1&0&0\end{pmatrix}\]
\end{itemize}
\end{example}

Like in the classical and the free case, we have a moment-cumulant formula for $\ee$-independence. Let us first describe its combinatorics.

\begin{definition}
For $i=(i(1),\ldots,i(k))\in\{1,\ldots,n\}^k$ we define $NC^\ee[i]$ as the set of all partitions $\pi\in P(k)$ such that
\begin{itemize}
\item $\pi\leq \ker i$, i.e. if $1\leq p<q\leq k$ are in the same block of $\pi$, then $i(p)=i(q)$,
\item and $\pi$ is \emph{$(\ee,i)$-noncrossing}, i.e. if there are indices\linebreak $1\leq p_1<q_1<p_2<q_2\leq k$ such that $p_1$ and $p_2$ are in a block $V_p$ of $\pi$ and $q_1$ and $q_2$ are in a block $V_q$ of $\pi$ with $V_p\neq V_q$, then $\ee_{i(p_1)i(q_1)}=1$.
\end{itemize}
\end{definition}

The idea is, that $NC^\ee[i]$ contains refinements of $\ker i$ which are allowed to have crossings only if the $\ee$-entry of the crossing is 1.

\begin{example}
\begin{itemize}
\item[(a)] For $\ee=\ee_{\comm}$, all kinds of crossings between blocks on different indices are allowed, but not for different blocks on the same index. Hence we have:
\[NC^{\ee_{\comm}}[i]=\prod_{V\in\ker i} NC(V)\]
\item[(b)] For $\ee=\ee_{\free}$, no crossings are allowed and hence:
\[NC^{\ee_{\free}}[i]=\{\pi\in NC(k)\;|\;\pi\leq\ker i\}\]
\end{itemize}
\end{example}

We now come to the moment-cumulant formula for $\ee$-independence found by the first author and Wysoczanski. We only formulate it for the situation of $\ee$-independent variables (rather than for algebras).

\begin{proposition}[{\cite[Thm. 4.2]{SpW}}]\label{MCFormula}
Let $x_1,\ldots,x_n\in A$ be $\ee$-independent and let $i=(i(1),\ldots, i(k))\in \{1,\ldots,n\}^k$. Then:
\[\phi(x_{i(1)}\cdots x_{i(k)})=\sum_{\pi\in NC^\ee[i]}\kappa_\pi(x_{i(1)},\ldots, x_{i(k)})\]
Here, $\kappa_\pi(x_{i(1)},\ldots, x_{i(k)})$ is the product of the free cumulants for each block.
\end{proposition}

\section{The $\ee$-sphere $S_{\RR,\ee}^{n-1}$}\label{SectEpsSphere}

The sphere in $\RR^n$ (also called the \emph{commutative sphere}) is given by:
\[S_{\RR}^{n-1}=\{(x_1,\ldots,x_n)\in\RR^n\;|\;\sum_i x_i^2=1\}\]
The algebra of continuous functions over it may be written as a universal $C^*$-algebra:
\[C(S_{\RR}^{n-1})=C^*(x_1,\ldots,x_n\;|\; x_i=x_i^*, \sum_i x_i^2=1, x_ix_j=x_jx_i \;\forall i,j)\]

A natural noncommutative analogue of the sphere is given by the \emph{(maximaly) noncommutative sphere} as introduced by Banica and Goswami \cite{BanGos}:
\[C(S_{\RR,+}^{n-1}):=C^*(x_1,\ldots,x_n\;|\; x_i=x_i^*, \sum_i x_i^2=1)\]
In the philosophy of noncommutative compact spaces, we may speak of the noncommutative sphere $S_{\RR,+}^{n-1}$ as a noncommutative compact space which is only defined via the algebra $C(S_{\RR,+}^{n-1})$ of (noncommutative) functions over it. Our next definition is an interpolation between the above spheres governed by the matrix $\ee$.

\begin{definition}
The \emph{$\ee$-sphere} $S_{\RR,\ee}^{n-1}$ is defined via the universal $C^*$-algebra:
\[C(S_{\RR,\ee}^{n-1}):=C^*(x_1,\ldots,x_n\;|\; x_i=x_i^*, \sum_i x_i^2=1, x_ix_j=x_jx_i, \textnormal{ if } \ee_{ij}=1)\]
\end{definition}

Note, that the commutative sphere $S_{\RR}^{n-1}$ is an $\ee$-sphere for the matrix $\ee_{\textnormal{comm}}$ of Example \ref{ExEps}(a). Moreover, the noncommutative sphere $S_{\RR,+}^{n-1}$ is an $\ee$-sphere with respect to $\ee_{\textnormal{free}}$ of Example \ref{ExEps}(b).
The $\ee$-sphere is noncommutative, if $\ee\neq\ee_{\comm}$, i.e. the $C^*$-algebra $C(S_{\RR,\ee}^{n-1})$ is noncommutative in this case. We prove it by using representations of the $\ee$-sphere which factor through the group $C^*$-algebras of certain Coxeter groups (see also \cite{BanCox,BanSurv}). 

\begin{definition}\label{DefCox}
Let $\mathbf F_n$ be the free group with $n$ generators $a_1,\ldots,a_n$ and denote by $\mathbf F_n^\ee$ the quotient of $\mathbf F_n$ by the relations $a_ia_j=a_ja_i$ if $\ee_{ij}=1$. Denote by $\mathbb Z_2^\ee$ the quotient of $\mathbf F_n^\ee$ by the relations $a_i^2=e$.
\end{definition}

By $\mathbb Z_2=\mathbb Z/2\mathbb Z$ we denote the cyclic group of order two. We may view $\mathbb Z_2^\ee$ as the quotient of the $n$-fold free product $\mathbb Z_2^{*n}$ by the relations $a_ia_j=a_ja_i$ if $\ee_{ij}=1$. It is a Coxeter group with the presentation:
\[\mathbb Z_2^\ee=\langle a_1,\ldots,a_n\;|\; (a_ia_j)^{m_{ij}}=e\rangle,\qquad 
m_{ij}=\begin{cases} 1 &\textnormal{ if } i=j\\2 &\textnormal{ if } \ee_{ij}=1\\\infty &\textnormal{ if }\ee_{ij}=0\end{cases}\]

\begin{example}
We have:
\begin{itemize}
\item[(a)] $\ZZ_2^\ee=\ZZ_2\times\ldots\times\ZZ_2$ for $\ee=\ee_{\comm}$ as in Example \ref{ExEps}(a)
\item[(b)] $\ZZ_2^\ee=\ZZ_2*\ldots*\ZZ_2$ for $\ee=\ee_{\free}$ as in Example \ref{ExEps}(b)
\item[(c)] $\ZZ_2^\ee=\ZZ_2^{\times n}*\ZZ_2^{*m}$ for $\ee$ as in Example \ref{ExEps}(c)
\item[(d)] $\ZZ_2^\ee=(\ZZ_2\times\ZZ_2)*(\ZZ_2\times\ZZ_2)$ for $\ee$ as in Example \ref{ExEps}(d)
\item[(e)] $\ZZ_2^\ee=(\ZZ_2*\ZZ_2)\times(\ZZ_2*\ZZ_2)$ for  $\ee$ as in Example \ref{ExEps}(e)
\item[(f)] For  $\ee$ as in Example \ref{ExEps}(f), we cannot write $\ZZ_2^\ee$ as an iteration of free and direct products.
\end{itemize}
\end{example}

The full group $C^*$-algebra associated to $\mathbb Z_2^\ee$ is the following universal $C^*$-algebra:
\[C^*(\mathbb Z_2^\ee)=C^*(z_1,\ldots, z_n\;|\; z_i=z_i^*, z_i^2=1, z_iz_j=z_jz_i \textnormal{ if } \ee_{ij}=1)\]

\begin{lemma}\label{RepSphere}
Let $H$ be a two-dimensional Hilbert space with orthonormal basis $e_1$ and $e_2$.
\begin{itemize}
\item[(a)] For $n=2$ and $\ee=\ee_{\free}$, the Coxeter group $\mathbb Z_2^\ee=\mathbb Z_2*\mathbb Z_2$ may be represented on $H$ by $\pi:C^*(\mathbb Z_2*\mathbb Z_2)\to B(H)$ defined as:
\[z_1\mapsto a:=\begin{pmatrix}-1&0\\0&1\end{pmatrix},\qquad
z_2\mapsto b:=\begin{pmatrix}0&1\\1&0\end{pmatrix}\]
Note that $a$ and $b$ do not commute.
\item[(b)] Let $n\in\NN$ and $\ee$ be arbitrary. Put $H_{ij}:=H$ for $1\leq i<j\leq n$. The representation $\sigma_{\ee}: C(\mathbb Z_2^\ee)\to B(\bigoplus_{i<j} H_{ij})$ given by
\[\sigma_{\ee}(z_k)_{|H_{ij}}=\begin{cases} 
a &\textnormal{ if } k=i\textnormal{ and }\ee_{ij}=0\\
b &\textnormal{ if } k=j\textnormal{ and }\ee_{ij}=0\\
\id_H & \textnormal{ otherwise }\end{cases}\]
is such that $\sigma_{\ee}(z_i)$ and $\sigma_{\ee}(z_j)$ commute if and only if $\ee_{ij}=1$.
\item[(c)] We may represent the $\ee$-sphere on the Coxeter group $\mathbb Z_2^\ee$ via:
\[\phi_\ee:C(S_{\RR,\ee}^{n-1})\to C^*(\mathbb Z_2^\ee),\qquad x_i\mapsto \frac{1}{\sqrt n}z_i\]
\end{itemize}
\end{lemma}
\begin{proof}
The proof is straightforward.
\end{proof}

\begin{theorem}\label{VglSphaeren}
We have $S_{\RR,\ee}^{n-1}\neq S_{\RR,\ee'}^{n-1}$ for $\ee\neq\ee'$ in the sense that there is no $^*$-isomorphism $C(S_{\RR,\ee}^{n-1})\to C(S_{\RR,\ee'}^{n-1})$ sending generators to generators.
Moreover, $C(S_{\RR,\ee}^{n-1})$ is noncommutative as soon as $\ee\neq\ee_{\comm}$.
\end{theorem}
\begin{proof}
If $\ee\neq \ee'$, we may find indices $i$ and $j$ such that $\ee_{ij}=1$ and $\ee'_{ij}=0$ (possibly after swapping the names for $\ee$ and $\ee'$). Assume that there is a $^*$-homomorphism $\psi:C(S_{\RR,\ee}^{n-1})\to C(S_{\RR,\ee'}^{n-1})$ mapping generators to generators. Composing it with $\sigma_{\ee'}\circ\phi_{\ee'}$ of the above lemma yields a contradiction, since $x_i$ and $x_j$ commute in $C(S_{\RR,\ee}^{n-1})$, but their images under $\sigma_{\ee'}\circ\phi_{\ee'}\circ\psi$ do not. Noncommutativity of $C(S_{\RR,\ee}^{n-1})$ for $\ee\neq\ee_{\comm}$ follows directly from applying $\sigma_\ee\circ\phi_\ee$.
\end{proof}

\begin{corollary}
Let $\ee\neq\ee_{\comm}$ and $\ee\neq\ee_{\free}$. Seen as noncommutative compact spaces, we have:
\[S_{\RR}^{n-1}\subsetneq S_{\RR,\ee}^{n-1}\subsetneq S_{\RR,+}^{n-1}\]
This means, we have surjective but non-injective $*$-homomorphisms
\[C(S_{\RR}^{n-1})\leftarrow C(S_{\RR,\ee}^{n-1}) \leftarrow C(S_{\RR,+}^{n-1})\]
sending generators to generators. 
\end{corollary}

\begin{remark}
The study of noncommutative spheres has a long history and goes back to
Podle\'s \cite{QSphere}; see also the work of Connes with Dubois-Violette \cite{Con1} or with Landi \cite{Con2}, also collected in the survey \cite{Landi}; see also \cite{BanCox,BanSurv} for recent expositions about noncommutative spheres and latest references. 

The extensive work of Banica on noncommutative spheres is to be highlighted, see amongst others \cite{BanGos,BanCox,BanSph,BanPoly,BanSurv}. However, his relations on the coordinates $x_i$ are mostly chosen in a uniform way \cite[Def. 2.2]{BanSph}, \cite[Def. 1.7]{BanCox} rather than as partial relations; so our spheres appear to be new.
\end{remark}

\section{The $\ee$-orthogonal quantum group $O_n^\ee$}

The algebra of functions over the orthogonal group $O_n\subset M_n(\RR)$ can be viewed as the following universal $C^*$-algebra:
\[C(O_n)=C^*(u_{ij}, i,j=1,\ldots,n\;|\; u_{ij}=u_{ij}^*, u \textnormal{ is orthogonal}, R^{\textnormal{comm}})\]
Here:
\begin{align*}
u \textnormal{ is orthogonal} &\qquad \Longleftrightarrow\qquad &\sum_k u_{ik}u_{jk}=\sum_k u_{ki}u_{kj}=\delta_{ij}\\
R^{\textnormal{comm}}&\qquad \Longleftrightarrow\qquad & u_{ij}u_{kl}=u_{kl}u_{ij} \quad \forall i,j,k,l
\end{align*}

Wang \cite{WangOrth} defined a noncommutative analogue of it, the \emph{(free) orthogonal quantum group} $O_n^+$ given by:
\[C(O_n^+)=C^*(u_{ij}, i,j=1,\ldots,n\;|\; u_{ij}=u_{ij}^*, u \textnormal{ is orthogonal})\]
For an introduction to compact matrix quantum groups, we refer to the original articles by Woronowicz \cite{WoCMQG,WoRemark} or the books \cite{Nesh,Tim}. In the sequel, the tensor product of $C^*$-algebras is always with respect to the minimal tensor product.

\begin{definition}\label{DefEspOn}
We define the \emph{$\ee$-orthogonal quantum group} $O_n^\ee$ via the following universal $C^*$-algebra
\[C(O_n^\ee)=C^*(u_{ij}, i,j=1,\ldots,n\;|\; u_{ij}=u_{ij}^*, u \textnormal{ is orthogonal}, R^{\ee}),\]
where the relations $R^\ee$ are defined by:
\[u_{ik}u_{jl}=\begin{cases}
u_{jl}u_{ik}&\text{if  $\ee_{ij}=1$ and $\ee_{kl}=1$}\\
u_{jk}u_{il}&\text{if  $\ee_{ij}=1$ and $\ee_{kl}=0$}\\
u_{il}u_{jk}&\text{if  $\ee_{ij}=0$ and $\ee_{kl}=1$}
\end{cases}\]
\end{definition}

We refer to Proposition \ref{PropRefinedOn} for further relations which are implied by the above ones.

\begin{lemma}\label{LemOnQG}
The $\ee$-orthogonal quantum group $O_n^\ee$ is a quantum group indeed, i.e. the $C^*$-algebra $C(O_n^\ee)$ gives rise to a compact matrix quantum group in Woronowicz's sense.
\end{lemma}
\begin{proof}
According to Woronowicz's axioms, all we have to prove is that the map $\Delta:C(O_n^\ee)\to C(O_n^\ee)\otimes C(O_n^\ee)$ with  $u_{ij}\mapsto u_{ij}':=\sum_k u_{ik}\otimes u_{kj}$ is a $*$-homomorphism, i.e. that the elements $u_{ij}'\in C(O_n^\ee)\otimes C(O_n^\ee)$ satisfy the relations of the $u_{ij}\in C(O_n^\ee)$. Self-adjointness and orthogonality of $u'$ is easy to see, so it remains to show that the relations $R^\ee$ are fulfilled for the $u_{ij}'$.
Consider first $\ee_{ij}=1$ and $\ee_{kl}=1$. Then we have:
\begin{align*}
u_{ik}'u_{jl}'
&=\sum_{p,r:\, \ee_{pr}=1} u_{ir}u_{jp}\otimes u_{rk}u_{pl}+\sum_{p,r:\,
\ee_{pr}=0} u_{ir}u_{jp}\otimes u_{rk}u_{pl}\\
&=\sum_{p,r:\, \ee_{pr}=1} u_{jp}u_{ir}\otimes u_{pl} u_{rk}+\sum_{p,r:\,
\ee_{pr}=0} u_{jr}u_{ip}\otimes u_{rl}u_{pk}\\
&=\sum_{p,r} u_{jp}u_{ir}\otimes u_{pl}u_{rk}\\
&=u_{jl}'u_{ik}'
\end{align*}
Consider now $\ee_{ij}=1$ and $\ee_{kl}=0$. Then we have:
\begin{align*}
u_{ik}'u_{jl}'
&=\sum_{p,r:\, \ee_{pr}=1} u_{ir}u_{jp}\otimes u_{rk}u_{pl}+\sum_{p,r:\,
\ee_{pr}=0} u_{ir}u_{jp}\otimes u_{rk}u_{pl}\\
&=\sum_{p,r:\, \ee_{pr}=1} u_{jp}u_{ir}\otimes u_{pk} u_{rl}+\sum_{p,r:\,
\ee_{pr}=0} u_{jr}u_{ip}\otimes u_{rk}u_{pl}\\
&=\sum_{p,r} u_{jp}u_{ir}\otimes u_{pk}u_{rl}\\
&=u_{jk}'u_{il}'
\end{align*}
The case $\ee_{ij}=0$ and $\ee_{kl}=1$ is similar.
\end{proof}

Again, it is easy to see that $O_n$ and $O_n^+$ fit into the framework of $\ee$-orthogonal quantum groups, using the matrices $\ee_{\textnormal{comm}}$ and $\ee_{\textnormal{free}}$ respectively. However, let us point out that the commutativity relations $R^{\textnormal{comm}}$ do \emph{not} imply $R^\ee$ for general $\ee$, i.e. $O_n^\ee$ is \emph{no interpolation} between $O_n$ and $O_n^+$. We say that a compact matrix quantum group $G$ is a \emph{quantum subgroup} of $H$ (writing $G\subset H$), if there is a surjective $^*$-homomorphism from $C(H)$ to $C(G)$ mapping generators to generators.

\begin{proposition}\label{PropNoInterpolation}
If $\ee\neq \ee_{\comm}$ and $\ee\neq \ee_{\free}$, we have:
\[O_n\not\subset O_n^\ee\subsetneq O_n^+\]
More general, we have in that case:
\[S_n\not\subset O_n^\ee\subsetneq O_n^+\]
\end{proposition}
\begin{proof}
Since the matrix $u$ in $O_n^\ee$ is orthogonal and has self-adjoint entries, we have  $O_n^\ee\subset O_n^+$. The inclusion is strict, since $S_n\subset O_n\subset  O_n^+$ but $S_n\not\subset O_n^\ee$, which we will prove next. We may find $i\neq j$ such that $\ee_{ij}=1$ (since $\ee\neq \ee_{\free}$), and $k\neq l$ such that $\ee_{kl}=0$ (since $\ee\neq\ee_{\comm}$). Let $\sigma\in S_n$ be a permutation with the properties:
\[\sigma(k)=i,\qquad\sigma(l)=j\]
The associated permutation matrix $a^\sigma\in M_n(\CC)$ is defined by $a^\sigma_{pq}=\delta_{p\sigma(q)}$. The evaluation map $\ev_\sigma: C(S_n)\to\CC$ is given by $\ev_\sigma(u_{pq})=\delta_{p\sigma(q)}$. Now, assume $S_n\subset O_n^\ee$, i.e. there is a surjective $^*$-homomorphism $\phi:C(O_n^\ee)\to C(S_n)$ sending generators to generators. Composing it with $\ev_\sigma$ yields the following contradiction:
\[1=\delta_{i\sigma(k)}\delta_{j\sigma(l)}=\ev_\sigma\circ\phi(u_{ik}u_{jl})
=\ev_\sigma\circ\phi(u_{jk}u_{il})=\delta_{j\sigma(k)}\delta_{i\sigma(l)}=0\]
\end{proof}

Next, we will show that different matrices $\ee$ give rise to different $\ee$-orthogonal quantum groups.

\begin{lemma}\label{RepOnEps}
We have the following $^*$-homomorphism:
\[\phi_\ee:C(O_n^\ee)\to C^*(\mathbb Z_2^\ee),\qquad u_{ij}\mapsto \delta_{ij}z_i\]
\end{lemma}
\begin{proof}
The existence of $\pi$ is due to the universal property.
\end{proof}

\begin{remark}
The preceding lemma  is due to the fact that the diagonal subgroup of $O_n^\ee$ is the Coxeter group $\mathbb Z_2^\ee$. The diagonal subgroup of a compact matrix quantum group $(A,u)$ is constructed as follows. First, take the quotient of $A$ by the relations $u_{ij}=0$ for $i\neq j$. If $u$ is a unitary, so are all $u_{ii}$ in the quotient and we thus obtain the group $C^*$-algebra $C^*(G)$ of some group $G$. This group is called the diagonal subgroup of $(A,u)$.
\end{remark}

\begin{proposition}\label{OnEpsNoncommutative}
We have $O_n^\ee\neq O_n^{\ee'}$ for $\ee\neq\ee'$ in the sense that there is no $^*$-isomorphism $C(O_n^\ee)\to C(O_n^{\ee'})$ sending generators to generators.
Moreover, $C(O_n^\ee)$ is noncommutative as soon as $\ee\neq\ee_{\comm}$.
\end{proposition}
\begin{proof}
The proof is similar to the one of Theorem \ref{VglSphaeren} using the maps $\phi_\ee$ of Lemma \ref{RepOnEps} rather than those of Lemma \ref{RepSphere}(c).
\end{proof}

Similarly to the well-known facts \cite{BanGos} that $O_n$ acts maximally on the commutative sphere $S_{\RR}^{n-1}$ and that $O_n^+$ acts maximally on the noncommutative sphere $S_{\RR,+}^{n-1}$, we observe that $O_n^\ee$ acts maximally on the $\ee$-sphere $S_{\RR,\ee}^{n-1}$. 

\begin{theorem}\label{Action}
The $\ee$-orthogonal quantum group $O_n^\ee$ acts on the $\ee$-sphere $S_{\RR,\ee}^{n-1}$ by the natural left and right actions
\[\alpha: C(S_{\RR,\ee}^{n-1})\to C(O_n^\ee)\otimes C(S_{\RR,\ee}^{n-1}),\qquad x_i\mapsto \sum_k  u_{ik}\otimes x_k\]
and:
\[\beta: C(S_{\RR,\ee}^{n-1})\to C(O_n^\ee)\otimes C(S_{\RR,\ee}^{n-1}),\qquad x_i\mapsto \sum_k  u_{ki}\otimes x_k\]
Moreover, $O_n^\ee$ is maximal with these actions in the sense that whenever $G$ is a compact matrix quantum group acting on $S_{\RR,\ee}^{n-1}$ in the above way, then $G\subset O_n^\ee$.
\end{theorem}
\begin{proof}

\emph{Step 1: Existence of $\alpha$ and $\beta$.}

We put $y_i:=\sum_k u_{ik}\otimes x_k$ and compute for  $\ee_{ij}=1$:
\begin{align*}
y_iy_j
&=\sum_{k,l:\ee_{kl}=1} u_{ik}u_{jl}\otimes x_kx_l+\sum_{k,l:\ee_{kl}=0} u_{ik}u_{jl}\otimes x_kx_l\\
&=\sum_{k,l:\ee_{kl}=1} u_{jl}u_{ik}\otimes x_lx_k+\sum_{k,l:\ee_{kl}=0} u_{jk}u_{il}\otimes x_kx_l\\
&=\sum_{k,l} u_{jl}u_{ik}\otimes x_lx_k\\
&=y_jy_i
\end{align*}
Furthermore, $y_i^*=y_i$ and $\sum_i y_i^2=1$ by an easy computation  using only the relations of $O_n^+$. Thus, $\alpha$ exists by the universal property. Likewise we deduce the existence of $\beta$.

\emph{Step 2: Maximality; definition of auxiliary maps.}

Now, let $G$ be another compact matrix quantum group acting on $S_{\RR,\ee}^{n-1}$ via:
\begin{align*}
&\alpha',\beta': C(S_{\RR,\ee}^{n-1})\to C(G)\otimes C(S_{\RR,\ee}^{n-1})\\
&\alpha'(x_i)= \sum_k  u_{ik}\otimes x_k, \qquad\beta'(x_i)= \sum_k  u_{ki}\otimes x_k
\end{align*}
For proving that there is a $^*$-homomorphism $C(O_n^\ee)\to C(G)$ sending generators to generators, we will make use of the following $^*$-homomorphisms. They arise from tensor products of the identity map $\id:C(G)\to C(G)$ with $^*$-homomorphisms from  $C(S_{\RR,\ee}^{n-1})$ to $\CC$ or to $C(S_{\RR,+}^1)$ respectively; we use the universal property of $C(S_{\RR,\ee}^{n-1})$ for the existence of the latter ones. We have:
\[\eta_k:C(G)\otimes C(S_{\RR,\ee}^{n-1})\to C(G), 
\qquad z\otimes x_i\mapsto \begin{cases} z &\textnormal{ if }i=k\\ 0 &\textnormal{ otherwise}\end{cases}\]
Moreover,  we have for $\ee_{kl}=1$:
\[\sigma_{kl}:C(G)\otimes C(S_{\RR,\ee}^{n-1})\to C(G), 
\qquad z\otimes x_i\mapsto \begin{cases} \frac{1}{\sqrt 2}z &\textnormal{ if }i=k\textnormal{ or }i=l\\
0 &\textnormal{ otherwise}\end{cases}\]
And for $\ee_{kl}=0$ with $k<l$:
\[\tau_{kl}:C(G)\otimes C(S_{\RR,\ee}^{n-1})\to C(G)\otimes C(S_{\RR,+}^1), 
\quad z\otimes x_i\mapsto \begin{cases} 
z\otimes x_1 &\textnormal{ if }i=k\\
z\otimes x_2 &\textnormal{ if }i=l\\
0 &\textnormal{ otherwise}\end{cases}\]

\emph{Step 3: Maximality; $u_{ij}=u_{ij}^*$ holds in $C(G)$.}

We observe that all generators of $C(G)$ are self-adjoint, by applying $\eta_k$ to the following equation:
\[\sum_k  u_{ik}\otimes x_k=\alpha'(x_i)=\alpha'(x_i)^*=\sum_k  u_{ik}^*\otimes x_k\]

\emph{Step 4: Maximality; the relations $R^\ee$ hold in $C(G)$.}

Let us compute:
\begin{align*}
\alpha'(x_ix_j)&=\sum_{k,l:\ee_{kl}=1} u_{ik}u_{jl}\otimes x_kx_l+\sum_{k,l:\ee_{kl}=0} u_{ik}u_{jl}\otimes x_kx_l\\
\alpha'(x_jx_i)&=\sum_{k,l:\ee_{kl}=1} u_{jk}u_{il}\otimes x_kx_l+\sum_{k,l:\ee_{kl}=0} u_{jk}u_{il}\otimes x_kx_l\\
\beta'(x_kx_l)&=\sum_{i,j:\ee_{ij}=1} u_{ik}u_{jl}\otimes x_ix_j+\sum_{i,j:\ee_{ij}=0} u_{ik}u_{jl}\otimes x_ix_j\\
\beta'(x_lx_k)&=\sum_{i,j:\ee_{ij}=1} u_{il}u_{jk}\otimes x_ix_j+\sum_{i,j:\ee_{ij}=0} u_{il}u_{jk}\otimes x_ix_j
\end{align*}
If $\ee_{ij}=1$, the terms $\alpha'(x_ix_j)$ and $\alpha'(x_jx_i)$ coincide. We apply $\tau_{kl}$ for $k<l$ and $\ee_{kl}=0$ to the equation $\alpha'(x_ix_j)=\alpha'(x_jx_i)$ and we obtain (where now $x_1,x_2\in C(S^1_{\RR,+})$):
\begin{align*}
&u_{ik}u_{jl}\otimes x_1x_2+u_{il}u_{jk}\otimes x_2x_1+u_{ik}u_{jk}\otimes x_1^2+u_{il}u_{jl}\otimes x_2^2\\
=&u_{jk}u_{il}\otimes x_1x_2+u_{jl}u_{ik}\otimes x_2x_1+u_{jk}u_{ik}\otimes x_1^2+u_{jl}u_{il}\otimes x_2^2
\end{align*}
By applying the maps $\eta_1$ and $\eta_2$, we obtain $u_{ik}u_{jk}=u_{jk}u_{ik}$ and $u_{il}u_{jl}=u_{jl}u_{il}$. By Lemma \ref{RepSphere}(a) we know $x_1x_2\neq x_2x_1$, so we finally obtain the following relations (including the case $k=l$):
\[\ee_{ij}=1,\ee_{kl}=0: \qquad\qquad u_{ik}u_{jl}=u_{jk}u_{il}\]
A similar argument using $\beta'$ yields:
\[\ee_{ij}=0,\ee_{kl}=1: \qquad\qquad u_{ik}u_{jl}=u_{il}u_{jk}\]
For $\ee_{kl}=1$, we have $x_kx_l=x_lx_k$, hence applying $\sigma_{kl}$ to the equation $\alpha'(x_ix_j)=\alpha'(x_jx_i)$ yields the relations:
\[\ee_{ij}=1,\ee_{kl}=1: \qquad\qquad u_{ik}u_{jl}+u_{il}u_{jk}=u_{jk}u_{il}+u_{jl}u_{ik}\]
Applying it on $\beta'(x_kx_l)=\beta'(x_lx_k)$ yields the relations:
\[\ee_{ij}=1,\ee_{kl}=1: \qquad\qquad u_{ik}u_{jl}+u_{jk}u_{il}= u_{il}u_{jk}+ u_{jl}u_{ik}\]
Combining these two relations, we obtain:
\[\ee_{ij}=1,\ee_{kl}=1: \qquad\qquad u_{ik}u_{jl}= u_{jl}u_{ik}\]

\emph{Step 5: Maximality; $u$ is orthogonal in $C(G)$.}

Using the relation $\sum_k x_k^2=1$ in $C(S_{\RR,\ee}^{n-1})$, we infer:
\begin{align*}
1\otimes 1
&=\sum_k\alpha'(x_k^2)\\
&=\sum_{kij} u_{ki}u_{kj}\otimes x_ix_j\\
&=\sum_{ij:i\neq j}\left(\sum_k u_{ki}u_{kj}\right)\otimes x_ix_j+
\sum_{i}\left(\sum_k u_{ki}^2\right)\otimes x_i^2
\end{align*}
Applying $\eta_i$ to this equation, we obtain
\[ \sum_k u_{ki}^2=1\qquad\forall i\]
and therefore:
\[\sum_{ij:i\neq j}\left(\sum_k u_{ki}u_{kj}\right)\otimes x_ix_j=0\]
If now $\ee_{ij}=0$, applying $\tau_{ij}$ and using $x_1x_2\neq x_2x_1$ in $C(S_{\RR,+}^1)$ yields:
\[\sum_k u_{ki}u_{kj}=0\]
If $\ee_{ij}=1$, then $x_ix_j=x_jx_i$ and using $\sigma_{ij}$ we deduce:
\[\sum_k u_{ki}u_{kj}+\sum_k u_{kj}u_{ki}=0\]
But as $u_{ki}u_{kj}=u_{kj}u_{ki}$, we infer:
\[\sum_k u_{ki}u_{kj}=0\]
Performing similar computations for $\beta'$, this proves orthogonality of $u$ and we may conclude that there is a $^*$-homomorphism $C(G)\to C(O_n^\ee)$ sending generators to generators; hence $G\subset O_n^\ee$.
\end{proof}

\section{The $\ee$-symmetric quantum group $S_n^\ee$}

Having defined an $\ee$-version of the orthogonal group $O_n$ by quotienting out the relations $R^\ee$ from $O_n^+$, it is natural to define $\ee$-versions of quantum  subgroups of $O_n^+$ in the same way. In order to do so for the symmetric (quantum) group, we observe that several natural relations are equivalent, as will be discussed in the sequel.

\subsection{The relations $\mathring R^\ee$}

Recall the relations $R^\ee$ from Definition \ref{DefEspOn}:
\[u_{ik}u_{jl}=\begin{cases}
u_{jl}u_{ik}&\text{if  $\ee_{ij}=1$ and $\ee_{kl}=1$}\\
u_{jk}u_{il}&\text{if  $\ee_{ij}=1$ and $\ee_{kl}=0$}\\
u_{il}u_{jk}&\text{if  $\ee_{ij}=0$ and $\ee_{kl}=1$}
\end{cases}\]

We now define some simpler relations.

\begin{definition}\label{DefRRing}
We define the relations $\mathring R^\ee$ by:
\[u_{ik}u_{jl}=\begin{cases}
u_{jl}u_{ik}&\text{if  $\ee_{ij}=1$ and $\ee_{kl}=1$}\\
0&\text{if  $\ee_{ij}=1$ and $\ee_{kl}=0$}\\
0&\text{if  $\ee_{ij}=0$ and $\ee_{kl}=1$}\end{cases}\]
In fact, we may also express them as:
\begin{align*}
&(\mathring R^\ee1) \quad u_{ik}u_{jl}=u_{jl}u_{ik}\textnormal{ if } \ee_{ij}=1 \textnormal{ and } \ee_{kl}=1\\
&(\mathring R^\ee2) \quad \delta_{\ee_{kl}=0}u_{ik}u_{jl}=\delta_{\ee_{ij}=0}u_{ik}u_{jl}
\end{align*}
\end{definition}

\begin{lemma}\label{LemREpsAeqRRing}
For any quantum subgroup $G_n\subset O_n^+$, we have:
\begin{itemize}
\item[(a)] The relations $\mathring R^\ee$ imply the relations $R^\ee$.
\item[(b)] If $u_{ik}u_{jk}=u_{ki}u_{kj}=0$ for all $i\neq j$ and all $k$, then the relations $R^\ee$ imply the relations $\mathring R^\ee$. 
\end{itemize}
\end{lemma}
\begin{proof}
(a) This follows since $\ee$ is symmetric.

(b) The relations $u_{ik}u_{jk}=0$ imply that the elements $u_{ik}^2$ are projections (and thus, the $u_{ik}$ are partial isometries). Indeed, we have $\sum_ju_{jk}^2=1$ by the orthogonality relations, thus:
\[u_{ik}^2=u_{ik}^2\sum_j u_{jk}^2=u_{ik}^4+\sum_{i\neq j}u_{ik}^2u_{jk}^2=u_{ik}^4\]
Recall that for $\ee_{ij}=1$ and $\ee_{kl}=0$, the relations $R^\ee$  imply:
\[u_{ik}u_{jl}=u_{jk}u_{il}\]
Multiplying this equation from the right with $u_{jl}^2$, we infer
\[u_{ik}u_{jl}=0,\]
since the projections $u_{jl}^2$ and $u_{il}^2$ are orthogonal to each other. The case $\ee_{ij}=0$ and $\ee_{kl}=1$ is similar.
\end{proof}

\begin{definition}\label{DefGEps}
For any quantum subgroup $G\subset O_n^+$ we define $G^\ee$ and $\mathring G^\ee$ by:
\[C(G^\ee):=C(G)/ \langle R^\ee\rangle,\qquad C(\mathring G^\ee):=C(G)/ \langle \mathring R^\ee\rangle\]
\end{definition}

For $\ee=\ee_{\free}$, we have $G^\ee=G$.

\begin{lemma}\label{LemRZwei}
Let $G$ be a compact matrix quantum group with fundamental unitary $u=(u_{ij})_{1\leq i,j\leq n}$.
\begin{itemize}
\item[(a)] If the relations ($\mathring R^\ee2$) hold for $u_{ij}$ in $C(G)$, then they also hold for $u_{ij}':=\sum_k u_{ik}\otimes u_{kj}\in C(G)\otimes C(G)$.
\item[(b)] If $G\subset O_n^+$ is a quantum subgroup of $O_n^+$, then also $G^\ee$ and $\mathring G^\ee$ are compact matrix quantum groups and we have $\mathring G^\ee\subset G^\ee\subset G$.
\end{itemize}
\end{lemma}
\begin{proof}
(a) We compute:
\begin{align*}
\delta_{\ee_{kl}=0}u_{ik}'u_{jl}'
&=\sum_{pq} u_{ip}u_{jq}\otimes\delta_{\ee_{kl}=0}u_{pk}u_{ql}\\
&=\sum_{pq} u_{ip}u_{jq}\otimes\delta_{\ee_{pq}=0}u_{pk}u_{ql}\\
&=\sum_{pq} \delta_{\ee_{ij}=0}u_{ip}u_{jq}\otimes u_{pk}u_{ql}\\
&=\delta_{\ee_{ij}=0}u_{ik}'u_{jl}'
\end{align*}

(b) In Lemma \ref{LemOnQG} we proved that the relations $R^\ee$ pass from $u_{ij}$ to $u_{ij}':=\sum_k u_{ik}\otimes u_{kj}$. Thus, $G^\ee$ is a compact matrix quantum group. As for $\mathring G^\ee$, we use (a) and Lemma \ref{LemREpsAeqRRing}.
\end{proof}

By the same argument as in Proposition \ref{PropNoInterpolation} we see that whenever $\ee\neq \ee_{\comm}$ and $\ee\neq \ee_{\free}$, we have:
\[S_n\not\subset G^\ee\subset O_n^+\]
This is particularly interesting, since the concept of easy quantum groups, as developed by Banica and Speicher \cite{BS}, provides a powerful approach for defining and studying quantum subgroups $G\subset O_n^+$, see also \cite{RW}. However, they come with the restriction  $S_n\subset G\subset O_n^+$. Thus, the $\ee$-versions of easy quantum groups are a further step in the 
direction of understanding all quantum subgroups of $O_n^+$. 

\subsection{Definition of $S_n^\ee$}

For $S_n^+$ the quotient by $R^\ee$ coincides with the one by $\mathring R^\ee$, by Lemma \ref{LemREpsAeqRRing}. Hence, we define the $\ee$-symmetric group $S_n^\ee$ as follows.

\begin{definition}
The \emph{$\ee$-symmetric group} $S_n^\ee$ is given by the quotient of $S_n^+$ by the relations $\mathring R^\ee$, i.e.:
\[C(S_n^\ee)= C^*(u_{ij}\;|\; u_{ij}=u_{ij}^*=u_{ij}^2,\sum_ku_{ik}=\sum_ku_{kj}=1\;\forall i,j\textnormal{ and } \mathring R^\ee)\]
\end{definition}

Viewing  $\ee\in M_n(\{0,1\})$ as the adjacency matrix of an undirected graph $\Gamma_\ee$, we observe that our definition of $S_n^\ee$ coincides with the one of a quantum automorphism group of $\Gamma_\ee$ given by Bichon \cite{BicQAG, BicQAGZwei}, see Proposition \ref{PropVglBB}. In this sense, we may justify the definition $S_n^\ee$ intrinsicly, i.e. as the quantum symmetry of some quantum space; exactly like we motivated our definition of $O_n^\ee$ as the quantum symmetry of the $\ee$-sphere. There is another definition of a quantum automorphism group of a graph given by Banica \cite{BanQAG}. We denote it by $S_n^{\Gamma_\ee}$ in order to keep the notations used in this article  consistent.

\begin{definition}[{\cite{BanQAG}}]\label{DefAut}
Given an undirected graph $\Gamma_\ee$ with adjacency matrix $\ee\in M_n(\{0,1\})$, its \emph{quantum automorphism group} $S_n^{\Gamma_\ee}$ is defined via:
\[C(S_n^{\Gamma_\ee}):=  C^*(u_{ij}\;|\; u_{ij}=u_{ij}^*=u_{ij}^2,\sum_ku_{ik}=\sum_ku_{kj}=1\;\forall i,j\textnormal{ and } u\ee =\ee u)\]
More explicitely, $u\ee=\ee u$ may be expressed as:
\[\sum_k \delta_{\ee_{kl}=1}u_{ik}=\sum_j \delta_{\ee_{ij}=1}u_{jl}\]
\end{definition}

\begin{lemma}\label{LemRelSnEps}
Let $G\subset S_n^+$ be a quantum subgroup of $S_n^+$. The following relations are equivalent:
\begin{itemize}
\item[(i)] The partial relations of $R^\ee$: $u_{ik}u_{jl}=u_{jk}u_{il}$ and $u_{ki}u_{lj}=u_{kj}u_{li}$ if  $\ee_{ij}=1$ and $\ee_{kl}=0$.
\item[(ii)] The partial relations $(\mathring R^\ee2)$ of $\mathring R^\ee$: $\delta_{\ee_{kl}=0}u_{ik}u_{jl}=\delta_{\ee_{ij}=0}u_{ik}u_{jl}$.
\item[(iii)] The relations $u\ee=\ee u$:
$\sum_k \delta_{\ee_{kl}=1}u_{ik}=\sum_j \delta_{\ee_{ij}=1}u_{jl}$.
\end{itemize}
\end{lemma}
\begin{proof}
The equivalence of (i) and (ii) follows from the proof of Lemma \ref{LemREpsAeqRRing}. We now prove that (ii) implies (iii). From (ii) we infer:
\[\sum_{k,j}\delta_{\ee_{kl}=1}\delta_{\ee_{ij}=0}u_{ik}u_{jl}=0 \quad\textnormal{and}\quad
\sum_{k,j}\delta_{\ee_{kl}=0}\delta_{\ee_{ij}=1}u_{ik}u_{jl}=0\]
Thus, using $\sum_j u_{jl}=\sum_k u_{ik}=1$, we have:
\[\sum_k \delta_{\ee_{kl}=1}u_{ik}
=\sum_{k,j}\delta_{\ee_{kl}=1}\delta_{\ee_{ij}=1}u_{ik}u_{jl}
=\sum_{j}\delta_{\ee_{ij}=1}u_{jl}\]
Conversely, assume that (iii) holds. We thus have:
\[\sum_{k'} \delta_{\ee_{k'l}=0}u_{ik'}=1-\sum_{k'} \delta_{\ee_{k'l}=1}u_{ik'}
=1- \sum_{j'} \delta_{\ee_{ij'}=1}u_{j'l}=\sum_{j'} \delta_{\ee_{ij'}=0}u_{j'l}\]
Furthermore, $\ee_{kl}=0$ and $\ee_{k'l}=1$ implies $k\neq k'$ and hence $u_{ik}u_{ik'}=0$.
Likewise we see $\delta_{\ee_{ij'}=0}\delta_{\ee_{ij}=1}u_{j'l}u_{jl}=0$. Therefore:
\begin{align*}
\delta_{\ee_{kl}=0}u_{ik}u_{jl}
&=\sum_{k'}\delta_{\ee_{kl}=0}\delta_{\ee_{k'l}=0}u_{ik}u_{ik'}u_{jl}
+\sum_{k'}\delta_{\ee_{kl}=0}\delta_{\ee_{k'l}=1}u_{ik}u_{ik'}u_{jl}\\
&=\sum_{k'}\delta_{\ee_{kl}=0}\delta_{\ee_{k'l}=0}u_{ik}u_{ik'}u_{jl}\\
&=\sum_{j'} \delta_{\ee_{kl}=0}\delta_{\ee_{ij'}=0}u_{ik}u_{j'l}u_{jl}\\
&=\sum_{j'} \delta_{\ee_{kl}=0}\delta_{\ee_{ij'}=0}\delta_{\ee_{ij}=0}u_{ik}u_{j'l}u_{jl}
\end{align*}
On the other hand:
\begin{align*}
\delta_{\ee_{ij}=0}u_{ik}u_{jl}
&=\sum_{j'}\delta_{\ee_{ij}=0}\delta_{\ee_{ij'}=0}u_{ik}u_{j'l}u_{jl}\\
&=\sum_{k'}\delta_{\ee_{ij}=0}\delta_{\ee_{k'l}=0}u_{ik}u_{ik'}u_{jl}\\
&=\sum_{k'}\delta_{\ee_{ij}=0}\delta_{\ee_{k'l}=0}\delta_{\ee_{kl}=0}u_{ik}u_{ik'}u_{jl}
\end{align*}
This proves that (ii) holds. 
\end{proof}

The previous lemma and the next proposition comparing the two different definitions of quantum automorphism groups of \cite{BicQAG} and \cite{BanQAG} may also be found in \cite[Sect. 3.1]{QAGThesis}.

\begin{proposition}\label{PropVglBB}
The $\ee$-symmetric quantum group $S_n^\ee$ coincides with the quantum automorphism group of $\Gamma_\ee$ as defined by Bichon, and it is a quantum subgroup of the quantum automorphism group $S_n^{\Gamma_\ee}$ as defined by Banica.
\end{proposition}
\begin{proof}
The relations (3.2) of Theorem 3.2 in \cite{BicQAG} are equivalent to ($\mathring R^\ee2$) and his relations (3.3) are equivalent to ($\mathring R^\ee1$). His relations (3.4) follow from (3.2) and (3.3) using for $\ee_{kl}=1$:
\[\sum_{ij}\delta_{\ee_{ij}=1}u_{ik}u_{jl}=
\sum_{ij}u_{ik}u_{jl}
=\left(\sum_{i}u_{ik}\right)\left(\sum_ju_{jl}\right)=1\]
 The assertion $S_n^\ee\subset S_n^{\Gamma_\ee}$ follows from the previous lemma. 
\end{proof}

\subsection{Noncommutativity of $C(S^\ee_n)$}

Observe that the $C^*$-algebras \linebreak$C(G^\ee)$ may collapse to something very small and they might be commutative. This depends on the particular choice of the matrix $\ee$ as may be seen in the next two examples of  $S_n^\ee$. 
Note that when Banica, Bichon and others investigate graphs which have no quantum symmetry \cite{BanBicQAG, BBGNoSym,  BanSurvey}, this is exactly the same question: They ask whether or not $C(S_n^{\Gamma_\ee})$ is commutative. Such investigations and concrete examples may also be found in \cite[Thm. 5.6.1 and Thm. 6.4.1]{QAGThesis}.

Recall from Section \ref{SectEpsSphere} that we may view the full group $C^*$-algebra associated to $\mathbb Z_2*\mathbb Z_2$ as a universal $C^*$-algebra. We now give a well-known alternative presentation.

\begin{lemma}\label{LemPQ}
The following universal $C^*$-algebras are isomorphic and noncommutative.
\begin{itemize}
\item[(a)] $C^*(\mathbb Z_2 *\mathbb Z_2)=C^*(z_1,z_2,1\;|\; z_i=z_i^*, z_i^2=1, i=1,2)$
\item[(b)] $C^*(p,q,1\;|\; p=p^*=p^2,q=q^*=q^2)$
\end{itemize}
\end{lemma}
\begin{proof}
The isomorphism between (a) and (b) is given by:
\[z_1\mapsto 2p-1,\qquad z_2\mapsto 2q-1,\qquad 1\mapsto 1\]
Noncommutativity follows from Lemma \ref{RepSphere}(a).
\end{proof}

The following example has also been treated by Bichon in \cite[Prop. 3.3]{BicQAG}.

\begin{example}\label{ExKey}
Let $\ee\in M_4(\{0,1\})$ be as in Example \ref{ExEps}(d) given by:
\[\ee=\begin{pmatrix}0&1&0&0\\1&0&0&0\\0&0&0&1\\0&0&1&0\end{pmatrix}\]
The $C^*$-algebra $C(S_4^\ee)$  is noncommutative and hence $S_4^\ee$ is a quantum group which is not a group. 
\end{example}
\begin{proof}
The following matrix (using Lemma \ref{LemPQ}) in $M_4(C^*(\mathbb Z_2*\mathbb Z_2))$ gives rise to a representation of $C(S_4^\ee)$ as may be verified directly:
\[\begin{pmatrix}p&1-p&0&0\\1-p&p&0&0\\0&0&q&1-q\\0&0&1-q&q\end{pmatrix}\]
We thus have a surjection of $C(S_4^\ee)$  onto $C^*(\mathbb Z_2*\mathbb Z_2)$ which proves that $C(S_4^\ee)$ is noncommutative.
\end{proof}

\begin{example}\label{ExConverse}
Let $\ee\in M_4(\{0,1\})$  be as in Example \ref{ExEps}(e) given by:
\[\ee=\begin{pmatrix}0&0&1&1\\0&0&1&1\\1&1&0&0\\1&1&0&0\end{pmatrix}\]
The $C^*$-algebra $C(S_4^\ee)$  is commutative, and hence $S_4^\ee$ is a group with $S_4^\ee\subset S_4$. The group is computed explicitly in Example \ref{ExTn}(b).
\end{example}
\begin{proof}
Let $i,k,j,l\in\{1,2,3,4\}$. We will show that $u_{ik}$ and $u_{jl}$ commute.

\emph{Case 1: $\ee_{ij}=1$ and $\ee_{kl}=1$.} Then $u_{ik}$ and $u_{jl}$ commute due to the relations $R^\ee$.

\emph{Case 2: $\ee_{ij}=1$ and $\ee_{kl}=0$.} By the definition of the matrix $\ee$, there are two indices $p\neq q$ such that $\ee_{kp}=\ee_{kq}=1$ and by Case 1, we know that $u_{ik}$ commutes with $u_{jp}$ as well as with $u_{jq}$. Moreover, $u_{ik}$ commutes with $u_{jk}$ since their product is zero if $i\neq j$ as the projections in each row and each column are orthogonal to each other. Now, since $\ee_{kk}=0$, we have $p\neq k$ and $q\neq k$ showing  that $u_{ik}$ commutes with three entries of the $j$-th row. As the fourth entry may be expressed as a linear combination of 1 and the other three entries (recall that we have $\sum_m u_{jm}=1$), we infer that $u_{ik}$ commutes with all entries of the $j$-th row, in particular with $u_{jl}$.

\emph{Case 3: $\ee_{ij}=0$ and $\ee_{kl}=1$.} The argument in Case 2 is symmetric.

\emph{Case 4: $\ee_{ij}=0$ and $\ee_{kl}=0$.} Again we use the fact that there are two indices $p\neq q$ such that $\ee_{kp}=\ee_{kq}=1$, which by Case 3 yields that  $u_{ik}$ commutes with $u_{jp}$ as well as with $u_{jq}$. Now,  $u_{ik}$ commutes with $u_{jk}$ for any $i$ and $j$ and we conclude as above that $u_{ik}$ and $u_{jl}$ commute.

We conclude that $C(S_4^\ee)$  is commutative, and hence $S_4^\ee\subset S_4$.
\end{proof}

The above two examples show that it depends on the choice of $\ee$ whether $C(S_4^\ee)$ is commutative or not. However, recall from Proposition \ref{OnEpsNoncommutative} that $C(O_4^\ee)$ is noncommutative in both examples.

\section{The commutative version of $S_n^\ee$: the group $T_n^\ee$}

Given the fact that $S_n^\ee$ may be a group in certain cases, it might be interesting to determine it. We now associate a subgroup of $S_n$ to any $S_n^\ee$, regardless whether $S_n^\ee$ is a group or not. It is the commutative version of $S_n^\ee$.

\begin{definition}\label{DefTn}
For a given $\ee\in M_n(\{0,1\})$ we define the following subgroup of $S_n$:
\[T_n^\ee:=\{\sigma\in S_n\;|\; \sigma\ee\sigma^{-1}=\ee\}\subset S_n\]
\end{definition}

It is nothing but the automorphism group of the graph $\Gamma_\ee$, since any permutation $\sigma$ with $\sigma\ee\sigma^{-1}=\ee$ is a bijection between the vertices of the graph such that $i$ and $j$ form an edge of the graph if and only if $\sigma(i)$ and $\sigma(j)$ do. Thus, if a graph has no quantum symmetries in the sense of \cite{BanBicQAG}, then $S_n^{\Gamma_\ee}=S_n^\ee=T_n^\ee$.

It turns out that while easy quantum groups are quantum subgroups of $O_n^+$ with the restriction of containing $S_n$, their $\ee$-versions come with the restriction of containing $T_n^\ee$. 

\begin{proposition}
Viewed as a quantum group, $T_n^\ee$ arises from the quotient $C(S_n)/\langle R^\ee\rangle$, or, in other words, as the quotient of $C(S_n^\ee)$ by the commutativity of all generators $u_{ij}$.
Hence, for all quantum groups $S_n\subset G\subset O_n^+$ we have:
\[T_n^\ee\subset G^\ee\subset O_n^+\]
If $S_n^\ee$ is a group, then $S_n^\ee=T_n^\ee$.
\end{proposition}
\begin{proof}
Since the $C^*$-algebra $C(S_n)/\langle R^\ee\rangle$  carries a compact matrix quantum group structure and since it is commutative, it is isomorphic to $C(H)$, where $H$ is some subgroup of $S_n$. It is given by all permutation matrices $a^\sigma\in S_n$ satisfying the relations $R^\ee$, i.e. we have:
\begin{align*}
\textnormal{For  }\ee_{ij}=1\textnormal{ and }\ee_{kl}=0:&\quad
&\delta_{i\sigma(k)}\delta_{j\sigma(l)}
=a^\sigma_{ik}a^\sigma_{jl}
=a^\sigma_{jk}a^\sigma_{il}
=\delta_{j\sigma(k)}\delta_{i\sigma(l)}\\
\textnormal{For  }\ee_{ij}=0\textnormal{ and }\ee_{kl}=1:&\quad
&\delta_{i\sigma(k)}\delta_{j\sigma(l)}
=a^\sigma_{ik}a^\sigma_{jl}
=a^\sigma_{il}a^\sigma_{jk}
=\delta_{i\sigma(l)}\delta_{j\sigma(k)}
\end{align*}
Let $\sigma\in H$ be a permutation and let $k\neq l$. Put $i:=\sigma(k)$ and $j:=\sigma(l)$ and assume $\ee_{kl}=1$. Then $\ee_{ij}=1$ since we would have a contradiction otherwise resulting from the relations $R^\ee$. Likewise, $\ee_{kl}=0$ implies $\ee_{ij}=0$. We deduce that $H$ consists exactly of all permutations $\sigma\in S_n$ such that:
\[\ee_{\sigma(k)\sigma(l)}=1\qquad\Longleftrightarrow\qquad\ee_{kl}=1\]
Writing $\sigma\in S_n$ as the permutation matrix $a^\sigma\in M_n(\CC)$ with $a^\sigma_{pq}=\delta_{p\sigma(q)}$, we see that the $(k,l)$-th entry of $(a^\sigma)^{-1}\ee a^\sigma$ is exactly $\ee_{\sigma(k)\sigma(l)}$ which proves $H=T_n^\ee$.

Finally, the natural quotient map from $C(S_n^+)/\langle R^\ee\rangle$ to  $C(S_n)/\langle R^\ee\rangle$ is an isomorphism, if  $C(S_n^+)/\langle R^\ee\rangle$ is commutative. Thus $S_n^\ee=T_n^\ee$, if $C(S_n^\ee)$ is commutative.
\end{proof}

\begin{example}\label{ExTn}
We now study $T_n^\ee$ in  certain examples.
\begin{itemize}
\item[(a)] We have $T_n^\ee=S_n$ if and only if $\ee=\ee_{\comm}$ or $\ee=\ee_{\free}$. 
Indeed, by definition, $T_n^\ee$ consists of all possibilities to permute rows and columns with the same permutation, such that $\ee$ does not change. Now, if there are $i,j,k,l$ such that $\ee_{ij}=1$ and $\ee_{kl}=0$ with $k\neq l$, then the permutation $\sigma\in S_n$ with $\sigma(i)=k$ and $\sigma(j)=l$ is not contained in $T_n^\ee$. We conclude that $T_n^\ee$ is large in the extreme cases (maximally commutative and maximally noncommutative situations) and smaller otherwise.
\item[(b)] In Examples \ref{ExEps}(d) and \ref{ExKey} as well as in  Examples \ref{ExEps}(e) and \ref{ExConverse}, the group $T_4^\ee$ is given by the subgroup of $S_4$ generated by the transpositions $(1,2)$ and $(3,4)$ and the cyclic permutation $(1,2,3,4)$. It has eight elements. 

Check that permuting the first column to the $k$-th column implies a unique condition for where the second row is permuted to. We are then left with two possible choices for the permutation of the other indices, thus we have four times two possibilities in total.
\item[(c)] The following matrix has trivial group $T_6^\ee$:
\[\ee=\begin{pmatrix}
0&0&1&0&0&0\\
0&0&0&0&1&0\\
1&0&0&0&0&1\\
0&0&0&0&1&1\\
0&1&0&1&0&1\\
0&0&1&1&1&0
\end{pmatrix}\]
Indeed, observe that the number of units in a column gives the restriction that we may only permute the first and the second column, or the third and the fourth, and finally the fifth and the sixth. But permuting the first column to the second position, we would need to permute the third row to the fifth, which is not allowed by the above mentioned restriction. Continuing this argument, we infer that $T_6^\ee$ consists only of the neutral element.
\end{itemize}
\end{example}

\section{Intertwiners for a de Finetti theorem}\label{SectIntertwinersSn}

By Woronowicz's Tannaka-Krein result \cite{WoTK}, any compact matrix quantum group is completely determined by its intertwiners.  See for instance \cite{TW} for an introduction to intertwiners and Tannaka-Krein theory close to our setting. For the easy quantum groups of Banica and Speicher, the intertwiner space is spanned by maps which are indexed by partitions. Let us recall how we associate linear maps to partitions $\pi$ in $P(k,l)$ having $k$ upper points and $l$ lower points, see \cite{BS} for details. For multi indices $i=(i(1),\ldots,i(k))$ and $j=(j(1),\ldots,j(l))$ with entries from $\{1,\ldots,n\}$ we denote by $\ker(i,j)$ the partition in $P(k,l)$ obtained from connecting two points if and only if the entries of the multi index $(i,j)$ coincide. This definition is an extension of the definition of $\ker i$.

\begin{definition}[{\cite{BS}}]\label{DefTpi}
Let $n\in\NN$.  To a partition $\pi\in P(k,l)$ with $k,l\in\NN_0$, we associate the linear map:
\begin{align*}
&T_\pi:(\CC^n)^{\otimes k}\to (\CC^n)^{\otimes l}\\
&e_{i(1)}\otimes\ldots\otimes e_{i(k)}\mapsto\sum_{j(1),\ldots,j(l)}\delta_{\pi\leq\ker(i,j)}e_{j(1)}\otimes\ldots\otimes e_{j(l)}
\end{align*}
Here, $(\CC^n)^{\otimes 0}=\CC$, by convention. For $\pi\in P(l)=P(0,l)$, we have:
\begin{align*}
&T_\pi(1)=\sum_{j(1),\ldots,j(l)}\delta_{\pi\leq\ker(j)}e_{j(1)}\otimes\ldots\otimes e_{j(l)}\\
&(T_\pi)^*(e_{j(1)}\otimes\ldots\otimes e_{j(l)})=\delta_{\pi\leq\ker(j)}
\end{align*}
\end{definition}

In our situation, we need a further linear  map in order to describe our intertwiners.

\begin{definition}\label{DefRCross}
For $n\in\NN$ we define $R_{\crosspart}^1:(\CC^n)^{\otimes 2}\to(\CC^n)^{\otimes 2}$ by:
\[R^1_{\crosspart}(e_i\otimes e_j):=\delta_{{\ee_{ij}=1}}e_j\otimes e_i\]
\end{definition}

\begin{lemma}\label{LemRRingEpsIntertwiner}
Let $G\subset O_n^+$ be a compact matrix quantum group. 
The map $R^1_{\crosspart}$ is an intertwiner of $G$ if and only if the relations $\mathring R^\ee$ hold in $G$.
\end{lemma}
\begin{proof}
 We first compute:
\begin{align*}
u^{\otimes 2}R^1_{\crosspart}(e_l\otimes e_k)
&=u^{\otimes 2}\left(\delta_{\ee_{kl}=1}e_k\otimes e_l\right)\\
&=\sum_{i,j}\delta_{\ee_{kl}=1}u_{ik}u_{jl}\otimes e_i\otimes e_j
\end{align*}
And:
\begin{align*}
R^1_{\crosspart}u^{\otimes 2}(e_l\otimes e_k)
&=\sum_{i,j}u_{jl}u_{ik}\otimes R^1_{\crosspart}(e_j\otimes e_i)\\
&=\sum_{i,j}\delta_{\ee_{ij}=1}u_{jl}u_{ik}\otimes e_i\otimes e_j
\end{align*}
We infer that $R^1_{\crosspart}$ is an intertwiner (i.e. $R_{\crosspart}^1u^{\otimes 2}=u^{\otimes 2}R_{\crosspart}^1$) if and only if for all $i,j,k,l$:
\[\delta_{\ee_{kl}=1}u_{ik}u_{jl}=\delta_{\ee_{ij}=1}u_{jl}u_{ik}\]
These relations are equivalent to $\mathring R^\ee$ of Definition \ref{DefRRing}.
\end{proof}

In the next section, we will prove de Finetti theorems for $S_n^\ee$, $H_n^\ee$, $\mathring O_n^\ee$ and $\mathring B_n^\ee$, which all contain the intertwiner $R_{\crosspart}^1$ as an essential ingredient of our proof. For $\ee=\ee_{\free}$, these quantum groups are easy quantum groups in the sense of Banica and Speicher \cite{BS}. Their categories of partitions are as follows.

\begin{definition}
We define the following subsets of the set $P$ of all partitions.
\begin{itemize}
\item[(i)] $P_2$ is the set of all pair partitions, i.e. each block of any partition $\pi\in P_2$ consists of exactly two points.
\item[(ii)] $P_{1,2}$ is the set of all partitions $\pi\in P$, whose blocks consist either of one or of two points.
\item[(iii)] $P_{\even}$ consists of all partitions $\pi\in P$ whose blocks consist of an even number of points.
\end{itemize}
Let $\mathcal C(k)\subset P(k)$ be a set of partitions. We define for any multi index $i$ of length $k$:
\[NC_{\mathcal C}^\ee[i]:=\mathcal C(k)\cap NC^\ee[i]\]
\end{definition}

The category of partitions of $S_n^+$ is $NC$, the category of $H_n^+$ is $NC\cap P_{\even}$, the category of $B_n^+$ is $NC\cap P_{1,2}$ and the category of  $O_n^+$ is $NC\cap P_2$, see \cite{BS, Web}.
The sets of the above definition behave nicely with respect to taking subpartitions.

\begin{definition}
Let $\pi\in P(k)$ and $\sigma\in P(l)$ with $l\leq k$. Then $\sigma$ is a \emph{subpartition} of $\pi$, if
\begin{itemize}
\item[(i)] there are indices $1\leq p\leq q\leq k$ with $q-p+1=l$ such that $\pi$ restricted to the points $p,p+1,\ldots,q$ coincides with $\sigma$,
\item[(ii)] and no point $p\leq s\leq q$ of $\pi$ is in the same block as a point $1\leq t<p$ or $q<t\leq k$.
\end{itemize}
\end{definition}

\begin{lemma}\label{Block}
Let $\mathcal C\in\{P,P_2,P_{1,2},P_{\even}\}$, let $\pi\in\mathcal C$ and let $\sigma$ be a subpartition of $\pi$. Then $\sigma\in\mathcal C$ and also $\pi'\in\mathcal C$, where $\pi'$ is the partition obtained when removing $\sigma$ from $\pi$.
\end{lemma}
\begin{proof}
The conditions of $\mathcal C$ on the number of points in each block hold true for $\sigma$ and $\pi'$.
\end{proof}

We need the following technical lemma.

\begin{lemma}\label{LemExistenceOfL}
Let $\pi\in P(k)$ be a partition containing no non-trivial subpartitions and let $\pi$ consist of at least two blocks.
Then, there is an index  $1\leq l <k$ such that:
\begin{itemize}
\item[(i)] The point $l$ belongs to a block $V$ whereas $l+1$ belongs to $V'$, and the blocks $V$ and $V'$ cross.
\item[(ii)] We have $\min\{x\in V'\}<\min\{x\in V\}$.
\end{itemize}
\end{lemma}
\begin{proof}
Firstly, observe that no block of $\pi$ consists of a single point (otherwise it would form a subpartition) and that every block crosses with at least one other block (otherwise we would either find subpartitions between its legs, or the block would form a subpartition itself). Secondly, check that we may always find a block $V_p$ containing indices $p_1<p_2$ such that
\begin{itemize}
\item[(1)] there is an index $p_1<s<p_2$ whose block crosses with $V_p$,
\item[(2)] and there are no two indices $q_1,q_2$ in a block $V_q\neq V_p$ with $q_1<p_1<q_2<p_2$.
\end{itemize}
For instance, the block $V$ containing the point $1$ does the job, with $p_1:=1$ and $p_2:=\max\{x\in V\}$. 

Now, let $V_p$ and $p_1$, $p_2$ be  such that (1) and (2) are satisfied and $p_2-p_1$ is minimal. Then $l:=p_2-1$ is not in $V_p$ by minimality of $p_2-p_1$ and we have (ii) because of  (2).
Assume that (i) does not hold. Then, the block $V$ containing $l$ does not cross with $V_p$. Hence, there is at least a second point $s\in V$ with $p_1<s<l<p_2$ such that there is an index $s<t<l$ whose block crosses with $V$. Let $V_{p'}\neq V_p$ be the unique block containing indices $s$ and $u$ with $p_1<s<u<p_2$ such that there is an index $s<t<u$ whose block crosses with $V_{p'}$ and such that $\min\{x\in V_{p'}\}$ is minimal. Then, $p_1':=\min\{x\in V_{p'}\}$ and $p_2':=\max\{x\in V_{p'}\;|\; x<p_2\}$ satisfy (1) and (2), but $p'_2-p'_1<p_2-p_1$ contradicts the minimality assumption on $V_p$.
\end{proof}

We are ready to prove the crucial ingredient for our de Finetti theorem. Note that for $\ee=\ee_{\free}$, the proof is trivial.

\begin{proposition} \label{CorIntertwiner}
Let $k\in\NN$ and $\pi\in P(k)$. Let $G\subset O_n^+$ be a quantum subgroup of $O_n^+$ and let $\mathcal C\in\{P,P_2,P_{1,2},P_{\even}\}$. Let $R_{\crosspart}^1$ be an intertwiner of $G$ as well as the maps $T_\sigma$ for $\sigma\in\mathcal C\cap NC$.
\begin{itemize}
\item[(a)] Let $\pi\in\mathcal C$. Then the following map is an intertwiner of $G$:
\[M_\pi:(\CC^n)^{\otimes k}\to\CC,\qquad e_{i(1)}\otimes\ldots\otimes e_{i(k)}\mapsto \delta_{\pi\in NC_{\mathcal C}^\ee[i]}\]
\item[(b)]The following relations hold in $G$, for all $k\in\NN$ and all \linebreak $j(1),\ldots,j(k)\in\{1,\ldots,n\}$.
\[\sum_{i(1),\ldots,i(k)}\delta_{\pi\in NC_{\mathcal C}^\ee[i]}
u_{i(1)j(1)}\cdots u_{i(k)j(k)}=\delta_{\pi\in NC_{\mathcal C}^\ee[j]}\]
\end{itemize}
\end{proposition}
\begin{proof}
Let $\pi\in \mathcal C(k)$ be a partition. We now construct the linear map $M_\pi$ recursively from composing intertwiners of $G$. For doing so, we use the following algorithm to construct partitions $\pi_m\in P(k_m)$ and maps $M_m:(\CC^n)^{\otimes k_m}\to\CC$, for $m\in\NN_0$. 

\emph{Step 1: The algorithm for defining $M_\pi$.}

Let $\pi_0:=\pi$ and $k_0:=k$ and begin the algorithm with $m=0$.
\begin{itemize}
\item[(Case 1)] If there is a noncrossing subpartition $\sigma$ of $\pi_m$ on the indices $p,p+1,\ldots,q$ with $1\leq p\leq q\leq k_m$, we define
\[M_m:= \id^{\otimes p-1}\otimes T_{\sigma}^*\otimes\id^{\otimes k_m-q}\]
and let $\pi_{m+1}$ be the partition resulting from $\pi_m$ when removing $\sigma$. Put $k_{m+1}:=k_m-(q-p+1)$. If $\sigma=\pi_m$, terminate the algorithm after this step.  
\item[(Case 2)] If Case 1 does not apply, let $1\leq l <k_m$ be the smallest index such that:
\begin{itemize}
\item[$\bullet$] The point $l$ belongs to a block $V$ whereas $l+1$ belongs to $V'$, and the blocks $V$ and $V'$ cross.
\item[$\bullet$] We have $\min\{x\in V'\}<\min\{x\in V\}$.
\end{itemize}
We put:
\[M_m:=\id^{\otimes (l-1)}\otimes R^1_{\crosspart}\otimes \id^{\otimes (k_m-l-1)}\]
We define $\pi_{m+1}$ as the partition obtained from $\pi_m$ when swapping the legs on $l$ and $l+1$. We put $k_{m+1}:=k_m$.
\end{itemize}

An example of the algorithm can be found in Figure \ref{FigAlg}. Note that if $\pi_m$ does not contain a noncrossing subpartition, either $\pi_m$ or one of its subpartitions satisfies the assumptions of  Lemma \ref{LemExistenceOfL} which ensures the existence of an index $l$ as in Case 2.

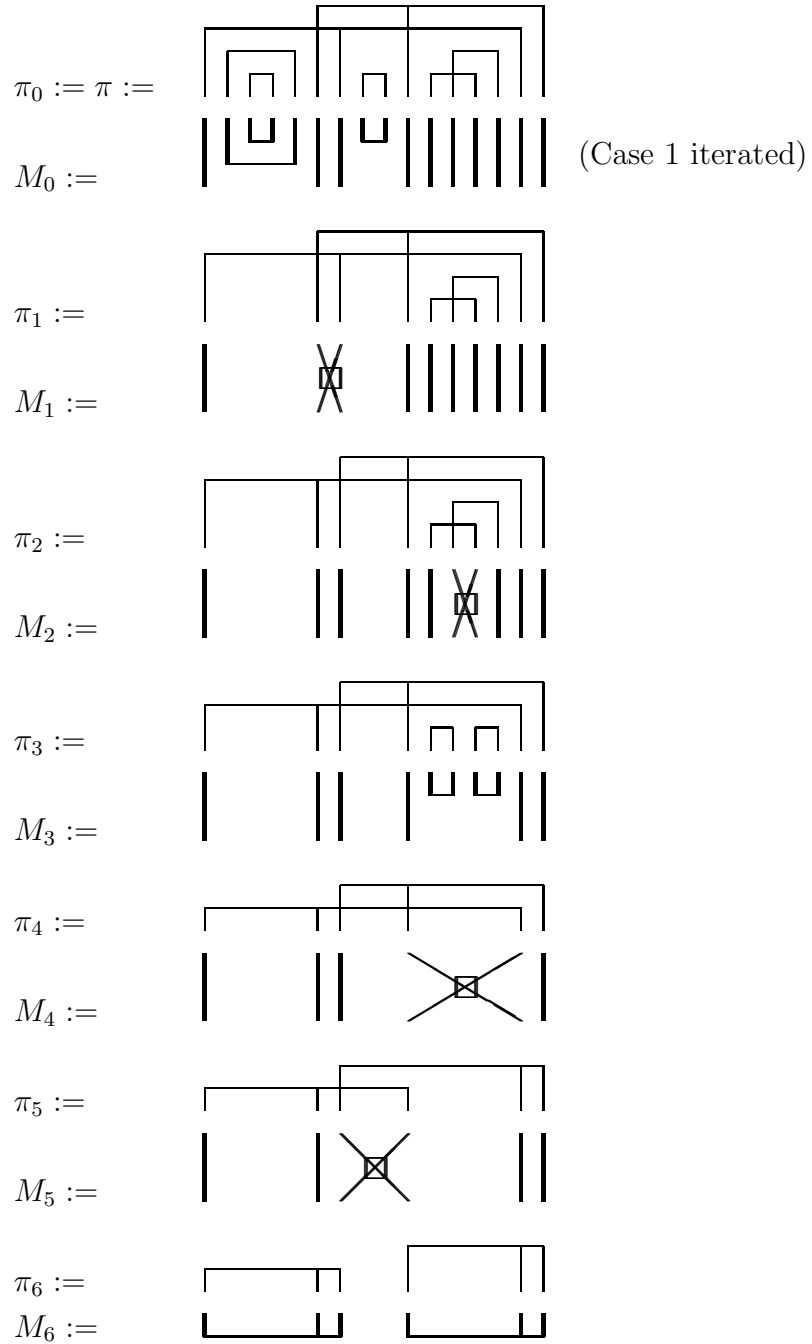
\begin{figure}
\setlength{\unitlength}{0.3cm}
\begin{center}
\begin{picture}(25,60)
\newsavebox{\XEins}\savebox{\XEins}
 {\begin{picture}(16,4)
   \put(0.5,-4){\line(1,3){1}}    
   \put(0.5,-1){\line(1,-3){1}}    
   \put(0.5,-3){$\square$}
   \end{picture}}
\newsavebox{\XDrei}\savebox{\XDrei}
 {\begin{picture}(16,4)
   \put(0.5,-4){\line(1,1){3}}    
   \put(0.5,-1){\line(1,-1){3}}    
   \put(1.5,-3){$\square$}
   \end{picture}}
\newsavebox{\XFuenf}\savebox{\XFuenf}
 {\begin{picture}(16,4)
   \put(0.5,-4){\line(5,3){5}}    
   \put(0.5,-1){\line(5,-3){5}}    
   \put(2.5,-3){$\square$}
   \end{picture}}
\newsavebox{\pNull}\savebox{\pNull}
 {\begin{picture}(16,4)
   \put(0,0){\uppartiii{3}{1}{7}{15}}
   \put(0,0){\uppartii{2}{2}{5}}   
   \put(0,0){\uppartii{1}{3}{4}}   
   \put(0,0){\uppartiii{4}{6}{10}{16}}   
   \put(0,0){\uppartii{1}{8}{9}}   
   \put(0,0){\uppartii{1}{11}{13}}   
   \put(0,0){\uppartii{2}{12}{14}}      
   \end{picture}}
\newsavebox{\pEins}\savebox{\pEins}
 {\begin{picture}(16,4)
   \put(0,0){\uppartiii{3}{1}{7}{15}}
   \put(0,0){\uppartiii{4}{6}{10}{16}}   
   \put(0,0){\uppartii{1}{11}{13}}   
   \put(0,0){\uppartii{2}{12}{14}}      
   \end{picture}}
\newsavebox{\pZwei}\savebox{\pZwei}
 {\begin{picture}(16,4)
   \put(0,0){\uppartiii{3}{1}{6}{15}}
   \put(0,0){\uppartiii{4}{7}{10}{16}}   
   \put(0,0){\uppartii{1}{11}{13}}   
   \put(0,0){\uppartii{2}{12}{14}}      
   \end{picture}}
\newsavebox{\pDrei}\savebox{\pDrei}
 {\begin{picture}(16,4)
   \put(0,0){\uppartiii{2}{1}{6}{15}}
   \put(0,0){\uppartiii{3}{7}{10}{16}}   
   \put(0,0){\uppartii{1}{11}{12}}   
   \put(0,0){\uppartii{1}{13}{14}}      
   \end{picture}}   
\newsavebox{\pVier}\savebox{\pVier}
 {\begin{picture}(16,4)
   \put(0,0){\uppartiii{1}{1}{6}{15}}
   \put(0,0){\uppartiii{2}{7}{10}{16}}   
   \end{picture}}      
\newsavebox{\pFuenf}\savebox{\pFuenf}
 {\begin{picture}(16,4)
   \put(0,0){\uppartiii{1}{1}{6}{10}}
   \put(0,0){\uppartiii{2}{7}{15}{16}}   
   \end{picture}}       
\newsavebox{\pSechs}\savebox{\pSechs}
 {\begin{picture}(16,4)
   \put(0,0){\uppartiii{1}{1}{6}{7}}
   \put(0,0){\uppartiii{2}{10}{15}{16}}   
   \end{picture}}      
\newsavebox{\MNull}\savebox{\MNull}
 {\begin{picture}(16,4)
   \put(0,0){\parti{3}{1}}
   \put(0,0){\partii{2}{2}{5}}   
   \put(0,0){\partii{1}{3}{4}} 
   \put(0,0){\parti{3}{6}}
   \put(0,0){\parti{3}{7}}        
   \put(0,0){\partii{1}{8}{9}}   
   \put(0,0){\parti{3}{10}}        
   \put(0,0){\parti{3}{11}}
   \put(0,0){\parti{3}{12}}        
   \put(0,0){\parti{3}{13}}
   \put(0,0){\parti{3}{14}}        
   \put(0,0){\parti{3}{15}}
   \put(0,0){\parti{3}{16}}        
   \end{picture}}
\newsavebox{\MEins}\savebox{\MEins}
 {\begin{picture}(16,4)
   \put(0,0){\parti{3}{1}}
   \put(6,0){\usebox{\XEins}}        
   \put(0,0){\parti{3}{10}}        
   \put(0,0){\parti{3}{11}}
   \put(0,0){\parti{3}{12}}        
   \put(0,0){\parti{3}{13}}
   \put(0,0){\parti{3}{14}}        
   \put(0,0){\parti{3}{15}}
   \put(0,0){\parti{3}{16}}        
   \end{picture}}
\newsavebox{\MZwei}\savebox{\MZwei}
 {\begin{picture}(16,4)
   \put(0,0){\parti{3}{1}}
   \put(0,0){\parti{3}{6}}
   \put(0,0){\parti{3}{7}}     
   \put(0,0){\parti{3}{10}}        
   \put(0,0){\parti{3}{11}}
   \put(12,0){\usebox{\XEins}}
   \put(0,0){\parti{3}{14}}        
   \put(0,0){\parti{3}{15}}
   \put(0,0){\parti{3}{16}}        
   \end{picture}}   
\newsavebox{\MDrei}\savebox{\MDrei}
 {\begin{picture}(16,4)
   \put(0,0){\parti{3}{1}}
   \put(0,0){\parti{3}{6}}
   \put(0,0){\parti{3}{7}}     
   \put(0,0){\parti{3}{10}}        
   \put(0,0){\partii{1}{11}{12}}
   \put(0,0){\partii{1}{13}{14}}  
   \put(0,0){\parti{3}{15}}
   \put(0,0){\parti{3}{16}}        
   \end{picture}}      
\newsavebox{\MVier}\savebox{\MVier}
 {\begin{picture}(16,4)
   \put(0,0){\parti{3}{1}}
   \put(0,0){\parti{3}{6}}
   \put(0,0){\parti{3}{7}}     
   \put(10,0){\usebox{\XFuenf}}
   \put(0,0){\parti{3}{16}}        
   \end{picture}}     
\newsavebox{\MFuenf}\savebox{\MFuenf}
 {\begin{picture}(16,4)
   \put(0,0){\parti{3}{1}}
   \put(0,0){\parti{3}{6}}
   \put(7,0){\usebox{\XDrei}}
   \put(0,0){\parti{3}{15}}   
   \put(0,0){\parti{3}{16}}        
   \end{picture}}        
\newsavebox{\MSechs}\savebox{\MSechs}
 {\begin{picture}(16,4)
   \put(0,0){\partiii{1}{1}{6}{7}}
   \put(0,0){\partiii{1}{10}{15}{16}}   
   \end{picture}}       
\put(0,55){$\pi_0:=\pi:=$}
\put(7,55){\usebox{\pNull}}
\put(0,51){$M_0:=$}
\put(7,55){\usebox{\MNull}}
\put(6.95,55){\usebox{\MNull}}
\put(7.05,55){\usebox{\MNull}}
\put(25,52){(Case 1 iterated)}
\put(0,45){$\pi_1:=$}
\put(7,45){\usebox{\pEins}}
\put(0,41){$M_1:=$}
\put(7,45){\usebox{\MEins}}
\put(6.95,45){\usebox{\MEins}}
\put(7.05,45){\usebox{\MEins}}
\put(0,35){$\pi_2:=$}
\put(7,35){\usebox{\pZwei}}
\put(0,31){$M_2:=$}
\put(7,35){\usebox{\MZwei}}
\put(6.95,35){\usebox{\MZwei}}
\put(7.05,35){\usebox{\MZwei}}
\put(0,26){$\pi_3:=$}
\put(7,26){\usebox{\pDrei}}
\put(0,22){$M_3:=$}
\put(7,26){\usebox{\MDrei}}
\put(6.95,26){\usebox{\MDrei}}
\put(7.05,26){\usebox{\MDrei}}
\put(0,18){$\pi_4:=$}
\put(7,18){\usebox{\pVier}}
\put(0,14){$M_4:=$}
\put(7,18){\usebox{\MVier}}
\put(6.95,18){\usebox{\MVier}}
\put(7.05,18){\usebox{\MVier}}
\put(0,10){$\pi_5:=$}
\put(7,10){\usebox{\pFuenf}}
\put(0,6){$M_5:=$}
\put(7,10){\usebox{\MFuenf}}
\put(6.95,10){\usebox{\MFuenf}}
\put(7.05,10){\usebox{\MFuenf}}
\put(0,2){$\pi_6:=$}
\put(7,2){\usebox{\pSechs}}
\put(0,0){$M_6:=$}
\put(7,2){\usebox{\MSechs}}
\put(6.95,2){\usebox{\MSechs}}
\put(7.05,2){\usebox{\MSechs}}
\end{picture}
\end{center}
\caption{An example of the algorithm in Thm. \ref{CorIntertwiner}.}\label{FigAlg}
\end{figure}

Moreover, the algorithm terminates since coming from Case 2, we will either be in Case 1 in the next step (reducing the length of the partition, or terminating) or we will be again in Case 2  successively pulling two crossing blocks side by side, which eventually brings us back to Case 1. Thus, we obtain a finite number of maps $M_0,\ldots,M_{t}$ and we put:
 \[M_\pi:=M_{t}\ldots M_0\]
 
\emph{Step 2: $M_\pi$ is an intertwiner of $G$.} 
 
It is easy to see, by Lemma \ref{Block}, that $\pi_m\in\mathcal C$ if and only if $\pi_{m+1}\in\mathcal C$. Hence, since $\pi_0=\pi$ is in $\mathcal C$, so are all $\pi_m$ and also all of their subpartitions. Therefore, by assumption on the intertwiner space of $G$, all maps  $M_m$ are intertwiners of $G$, and so is $M_\pi$.

\emph{Step 3: Proof of $M_\pi e_i=\delta_{\pi\in NC_{\mathcal C}^\ee[i]}$.}

We are left with proving the formula:
\[M_\pi(e_{i(1)}\otimes\ldots\otimes e_{i(k)})=\delta_{\pi\in NC_{\mathcal C}^\ee[i]}\]
Let us abbreviate $e_i:=e_{i(1)}\otimes\ldots\otimes e_{i(k)}$ for multi indices $i$. We are going to prove the following statement:
\begin{itemize}
\item[(*)] For all $0\leq m\leq t-1$ and for all multi indices $i$, we have $\pi_m\in NC_{\mathcal C}^\ee[i]$ if and only if $M_me_i\neq 0$ and $\pi_{m+1}\in NC_{\mathcal C}^\ee[j]$ for $e_j=M_me_i$.
\end{itemize}

Having proven (*), we infer inductively that $\pi=\pi_0\in NC_{\mathcal C}^\ee[i]$ if and only if  $\pi_{t}\in NC_{\mathcal C}^{\ee}[j]$ for $e_j=M_{t-1}\cdots M_0e_i\neq 0$. Since the above algorithm terminates at step $t$, the partition $\pi_t$ is noncrossing and hence $\pi_{t}\in NC_{\mathcal C}^{\ee}[j]$ if and only if $M_te_j=T_{\pi_t}^* e_j=1$ and $\pi_t\in\mathcal C$. We conclude $\pi\in NC_{\mathcal C}^\ee[i]$ if and only if $M_\pi e_i=1$. Since $M_\pi e_i$ only takes the values zero or one, this proves $M_\pi e_i=\delta_{\pi\in NC_{\mathcal C}^\ee[i]}$.

We are now going to prove (*). Let $0\leq m\leq t-1$ and let $i$ be a multi index. 

\emph{Step 4: Proof of (*) with $\pi_{m+1}$ resulting from Case 1 of the algorithm.}

Assume $\pi_m\in NC_{\mathcal C}^\ee[i]$. Then $\pi_m\leq \ker i$ implies that the noncrossing subpartition $\sigma$ of $\pi_m$ is less or equal to  the relevant section of the multi index $i$, and hence $M_me_i\neq 0$. Since $\pi_{m+1}$ arises as a restriction of $\pi_m$ while $j$ arises as a restriction of $i$, we also have $\pi_{m+1}\in NC_{\mathcal C}^\ee[j]$.

Conversely, let $\pi_{m+1}\in NC_{\mathcal C}^\ee[j]$ and $M_me_i\neq 0$. We have $\pi_m\leq \ker i$ by the following. If $p$ and $q$ are in the same block of $\pi_m$, then either both of them are in $j$ or none of them is, since $\sigma$ is a subpartition. In the first case, $\pi_{m+1}\leq\ker j$ implies $i(p)=i(q)$ whereas in the second, this is ensured by $M_me_i\neq 0$. Finally,  $\pi_m$ is $(\ee,i)$-noncrossing since any crossing in $\pi_m$ yields a crossing in $\pi_{m+1}$ whose $\ee$-entry is 1, because $\pi_{m+1}$ is $(\ee,j)$-noncrossing.

\emph{Step 5: Proof of (*) with $\pi_{m+1}$ resulting from Case 2 of the algorithm.}

In Case 2 of the algorithm, $\pi_{m+1}$ is obtained from $\pi_m$ by swapping the legs on $l$ and $l+1$. We have $M_me_i\neq 0$ if and only if $\ee_{i(l)i(l+1)}=1$. Moreover, assuming $\pi_m\in NC_{\mathcal C}^\ee[i]$, we obtain $\ee_{i(l)i(l+1)}=1$, since the blocks on $l$ and $l+1$ cross.  We may thus assume $\ee_{i(l)i(l+1)}=1$ throughout Step 5. We know that $j$ is of the form:
\[j=(i(1),\ldots,i(l-1),i(l+1),i(l),i(l+2),\ldots,i(k_m)\]

Assume $\pi_m\leq \ker i$. We want to prove $\pi_{m+1}\leq\ker j$. Let $p$ and $q$ be in the same block of $\pi_{m+1}$. If $\{p,q\}\cap\{l,l+1\}=\emptyset$, then $\pi_{m+1}$ and $\pi_m$ coincide on $p$ and $q$, i.e. $p$ and $q$ are also in the same block of $\pi_m$, implying $j(p)=i(p)=i(q)=j(q)$. On the other hand, if $\{p,q\}\cap\{l,l+1\}\neq\emptyset$, assume $p=l$. Then $q\neq l+1$ since $l$ and $l+1$ are in different blocks. Now, $l$ and $q$ being in the same block of $\pi_{m+1}$ implies that $l+1$ and $q$ are in the same block of $\pi_m$ and hence $j(p)=i(l+1)=i(q)=j(q)$. The other cases of $\{p,q\}\cap\{l,l+1\}\neq\emptyset$ are similar. Since the argument is symmetric, we just proved $\pi_m\leq \ker i$ if and only if $\pi_{m+1}\leq\ker j$.

Assume that $\pi_m$ is $(\ee,i)$-noncrossing. Then $\pi_{m+1}$ is $(\ee,j)$-noncrossing due to the following discussion. Let $p_1<q_1<p_2<q_2$ be points of $\pi_{m+1}$ such that $p_1,p_2\in V_p$ and $q_1,q_2\in V_q\neq V_p$. If $\{p_1,q_1,p_2,q_2\}\cap\{l,l+1\}=\emptyset$, we have $\ee_{j(p_1)j(q_1)}=1$ since $\pi_{m+1}$ coincides with $\pi_m$ on the relevant points. If $\{p_1,q_1,p_2,q_2\}\cap\{l,l+1\}=\{l\}$, then we are in the situation that some block $V$ of $\pi_{m+1}$ crosses with the block on $l$. This is equivalent to this block $V$ of $\pi_m$ crossing with the block on $l+1$ in $\pi_m$. We infer $\ee_{j(p_1)j(q_1)}=1$. We argue analoguously if $\{p_1,q_1,p_2,q_2\}\cap\{l,l+1\}=\{l+1\}$. Finally, $\{p_1,q_1,p_2,q_2\}\cap\{l,l+1\}=\{l,l+1\}$ implies $\ee_{i(p_1)i(q_1)}=\ee_{i(l)i(l+1)}=1$. Again, the argument is symmetric in $\pi_m$ and $\pi_{m+1}$.

We conclude that $\pi_m\in NC_{\mathcal C}^\ee[i]$ if and only if $\pi_{m+1}\in NC_{\mathcal C}^\ee[j]$, provided that $\ee_{i(l)i(l+1)}=1$.

(b) Finally, $M_\pi u^{\otimes k}=M_\pi$ yields the desired relations on the $u_{ij}$'s.
\end{proof}

\section{Symmetries of $\ee$-independence}

There are classical and noncommutative versions of de Finetti theorems characterizing independences by distributional symmetries. We recall when a distribution is invariant under a quantum group action. For details see \cite{SpK,BCS}.

\begin{definition}
Let $x_1,\ldots, x_n\in A$ be self-adjoint random variables in a noncommutative probability space $(A,\phi)$. 
\begin{itemize}
\item[(a)] Let $G\subset O_n^+$ be a compact matrix quantum group. We say that \emph{the distribution of $x_1,\ldots,x_n$ is invariant under $G$}, if for all $k\in\NN$ and all $j(1),\ldots,j(k)\in\{1,\ldots,n\}$:
\[\phi(x_{j(1)}\cdots x_{j(k)})1=
\sum_{i(1),\ldots,i(k)}\phi(x_{i(1)}\cdots x_{i(k)})u_{i(1)j(1)}\cdots u_{i(k)j(k)}\]
\item[(b)] We say that the variables $x_1,\ldots,x_n$ are \emph{identically distributed}, if we have for all $k\in\NN$ and all $1\leq i,j\leq n$:
\[\phi(x_i^k)=\phi(x_j^k)\]
\end{itemize}
\end{definition}

The above definition (a) is a natural extension of the notion of distributional invariance for groups. Indeed, for instance if $G=S_n$, evaluating the above equation at $\sigma\in S_n$ yields:
\begin{align*}
\phi(x_{j(1)}\cdots x_{j(k)})1
&=\sum_{i(1),\ldots,i(k)}\phi(x_{i(1)}\cdots x_{i(k)})\delta_{i(1)\sigma(j(1))}\cdots \delta_{i(k)\sigma(j(k))}\\
&=\phi(x_{\sigma(j(1))}\cdots x_{\sigma(j(k))})1
\end{align*}
This is just the well-known exchangeability. We now recall some existing de Finetti theorems.

\begin{proposition}[{\cite{SpK}}]\label{deFinettiClausRoland}
Let $(x_n)_{n\in\NN}$ be a sequence of selfadjoint random variables in a noncommutative $W^*$-probability space $(M,\phi)$ such that $M$ is generated by $x_n, n\in\NN$. The following holds true.
\begin{itemize}
\item[(a)] Suppose that the elements $x_n$ comute.

\noindent
The sequence $(x_n)_{n\in\NN}$ is conditionally independent and identically distributed if and only if its distribution is invariant under $(S_n)_{n\in\NN}$.
\item[(b)] The sequence $(x_n)_{n\in\NN}$ is conditionally free  and identically distributed  if and only if its distribution is invariant under $(S_n^+)_{n\in\NN}$.
\end{itemize}
\end{proposition}

The above theorem has been extended to other quantum groups by Banica, Curran and the first author \cite{BCS}:
\begin{itemize}
\item Invariance under $H_n^+$ adds the condition that the distribution of the variables is even.
\item Invariance under $B_n^+$ adds the condition that the distribution are semicircles with common mean and common variance.
\item Invariance under $O_n^+$ adds  the condition that the distribution are semicircles with mean zero and common variance.
\end{itemize}

In our framework, we do not have such a de Finetti theorem for the moment since we are lacking  an operator-valued version of $\ee$-indepence (needed to formulate what ``conditionally $\ee$-independent'' is supposed to mean). However, the equivalences of the above de Finetti theorems rely on a finite and purely algebraic version of one of the directions which we formulate here in the scalar-valued form.

\begin{proposition}[{\cite{BCS}}]
Let $x_1,\ldots,x_n$ be selfadjoint random variables in a noncommutative probability space $(A,\phi)$.
\begin{itemize}
\item[(a)] Suppose that the elements $x_1,\ldots,x_n$ commute.

\noindent
If $x_1,\ldots,x_n$ are independent and identically distributed, then their distribution is invariant under $S_n$.
\item[(b)] If $x_1,\ldots,x_n$ are free and identically distributed, then their distribution is invariant under $S_n^+$.
\end{itemize}
\end{proposition}

Again, this statement has versions for $H_n^+$, $B_n^+$ and $O_n^+$, \cite{BCS}.
We now prove an $\ee$-version containing the above proposition as a special case. For doing so, we need to define further quantum groups based on Definition \ref{DefGEps} and Lemma \ref{LemRZwei}.
Recall from Lemma \ref{LemREpsAeqRRing} that the relations $R^\ee$ and $\mathring R^\ee$ are equivalent in any quotient of $C(H_n^+)$. 

\begin{definition}
The \emph{$\ee$-hyperoctahedral group} $H_n^\ee$ is given by the quotient of $H_n^+$ by the relations $\mathring R^\ee$, i.e.:
\[C(H_n^\ee):= C^*(u_{ij}\;|\; u_{ik}u_{jk}=u_{ki}u_{kj}=0\;\forall i\neq j\textnormal{ and } \mathring R^\ee)\]
\end{definition}

For $\ee=\ee_{\free}$, the quantum group $H_n^\ee$ is nothing but the hyperoctahedral quantum group $H_n^+$ as defined by Banica, Bichon and Collins \cite{BBC}. 
As for the analogs of $O_n^+$ and $B_n^+$ in our de Finetti theorem, there is a little subtlety: The relations $R^\ee$ and $\mathring R^\ee$ are \emph{not} equivalent for general subgroups of $O_n^+$. Moreover, our Proposition \ref{CorIntertwiner} requires the existence of $R^1_{\crosspart}$ as an intertwiner. Therefore (see Lemma \ref{LemRRingEpsIntertwiner}), our de Finetti theorem is designed for quantum groups satisfying the relations $\mathring R^\ee$.

\begin{definition}\label{DefRRingOn}
We define the quantum group $\mathring O_n^\ee$ via the following universal $C^*$-algebra:
\[C(\mathring O_n^\ee)=C^*(u_{ij}, i,j=1,\ldots,n\;|\; u_{ij}=u_{ij}^*, u \textnormal{ is orthogonal}, \mathring R^{\ee})\]
We define the quantum group $\mathring B_n^\ee$ via the quotient of $C(\mathring O_n^\ee)$ by the relations $\sum_ku_{ik}=\sum_ku_{kj}=1$.
\end{definition}

Note that for  $\ee=\ee_{\free}$, we have $\mathring O_n^\ee= O_n^\ee=O_n^+$, but for $\ee=\ee_{\comm}$, we have $\mathring O_n^\ee=H_n \subsetneq O_n=O_n^\ee$, since in that case $u_{ik}u_{jk}=u_{ki}u_{kj}=0$ in $\mathring O_n^\ee$ for all $i\neq j$.

\begin{theorem}\label{deFinetti}
Let $\ee\in M_n(\{0,1\})$ and let $x_1,\ldots,x_n$ be selfadjoint random variables in a noncommutative probability space $(A,\phi)$ such that $x_ix_j=x_jx_i$ if $\ee_{ij}=1$.
Let $x_1,\ldots,x_n$ be $\ee$-independent and identically distributed.
\begin{itemize}
\item[(a)] Then their distribution is invariant under $S_n^\ee$.
\item[(b)] If their distribution is even (all odd free cumulants vanish), then it is invariant under $H_n^\ee$.
\item[(c)] If only free cumulants of blocks of size one or two are non-zero (i.e. if the distribution is a shifted semicircular), then the distribution is invariant under $\mathring B_n^\ee$.
\item[(d)] If only free cumulants of blocks of size two are non-zero (i.e. if the distribution is a centered semicircular), then the distribution is invariant under $\mathring O_n^\ee$.
\end{itemize}
\end{theorem}
\begin{proof}
Let $x_1,\ldots,x_n$ be $\ee$-independent and identically distributed, let $k\in\NN$ and let $j(1),\ldots,j(k)\in\{1,\ldots,n\}$.  \newline Let $(\mathcal C,G)\in \{(P,S_n^\ee),(P_{\even},H_n^\ee),(P_{1,2},\mathring B_n^\ee),(P_2,\mathring O_n^\ee)\}$ and let the cumulants of the distributions of $x_1,\ldots,x_n$ be according to the assumptions in (a), (b), (c) or (d) respectively.
 By the moment-cumulant formula (Proposition \ref{MCFormula}) and our assumptions on the cumulants we have:
\[\phi(x_{j(1)}\cdots x_{j(k)})=\sum_{\pi\in NC_{\mathcal C} ^\ee[j]}\kappa_\pi(x_{j(1)},\ldots, x_{j(k)})\]
For $\pi\in NC_{\mathcal C}^\ee[j]$, the cumulant $\kappa_\pi(x_{j(1)},\ldots,x_{j(k)})$ factorizes according to the blocks of $\pi$ (see \cite{SpW}) and on each such block the indices $j(l)$ coincide, since $\pi\leq \ker j$. Now, $x_1,\ldots,x_n$ are identically distributed, thus those single block cumulants do not depend on the index $j$, and hence nor does $\kappa_\pi$. Therefore, we put $\kappa_\pi:=\kappa_\pi(x_{j(1)},\ldots, x_{j(k)})$ for any $j$ with $\pi\leq \ker j$. We then compute, using Proposition \ref{CorIntertwiner}:
\begin{align*}
\sum_{i(1),\ldots,i(k)}&\phi(x_{i(1)}\cdots x_{i(k)})u_{i(1)j(1)}\cdots u_{i(k)j(k)}\\
&=\sum_{i(1),\ldots,i(k)}\left(\sum_{\pi\in NC_{\mathcal C}^\ee[i]}\kappa_\pi(x_{i(1)},\ldots, x_{i(k)})\right)
u_{i(1)j(1)}\cdots u_{i(k)j(k)}\\
&=\sum_{i(1),\ldots,i(k)}\left(\sum_{\pi\in P(k)}\delta_{\pi\in NC_{\mathcal C}^\ee[i]}\kappa_\pi\right)
u_{i(1)j(1)}\cdots u_{i(k)j(k)}\\
&=\sum_{\pi\in P(k)}\kappa_\pi\sum_{i(1),\ldots,i(k)}\delta_{\pi\in NC_{\mathcal C}^\ee[i]}
u_{i(1)j(1)}\cdots u_{i(k)j(k)}\\
&=\sum_{\pi\in P(k)}\kappa_\pi\delta_{\pi\in NC_{\mathcal C}^\ee[j]}\\
&=\phi(x_{j(1)}\cdots x_{j(k)})
\end{align*}
Thus, the distribution of $x_1,\ldots,x_n$ is invariant under $G$.
\end{proof}

\begin{remark}
Distributional invariance under $S_n^+$ is also called \emph{quantum exchangeability}, while invariance under $S_n$ is called \emph{exchangeability}. The former one implies the latter one \cite{SpK}.  In our case, invariance under $S_n^\ee$ (which we could call \emph{$\ee$-quantum exchangeability}) does \emph{not} imply exchangeability. We only obtain invariance under $T_n^\ee$ (which we could call \emph{$\ee$-exchangeability}), i.e.:
\[\phi(x_{j(1)}\cdots x_{j(k)})=\phi(x_{\sigma(j(1))}\cdots x_{\sigma(j(k))})\qquad\forall\sigma\in T_n^\ee\]
Since exchangeability of variables implies that they are identically distributed, and since we only have $\ee$-exchangeability in our case, it is likely that one can weaken the assumptions of our de Finetti theorem.
\end{remark}

\section{Partition calculus and intertwiners}

As already mentioned in Section \ref{SectIntertwinersSn}, Woronowicz's concept of intertwiner spaces gives a Tannaka-Krein type approach to compact matrix quantum groups. Moreover, the calculus with intertwiners provides a fairly easy way of deducing $C^*$-algebraic relations from others. The concept of easy quantum groups as introduced by Banica and Speicher \cite{BS} offers yet another simplification of the intertwiner calculus: For any easy quantum group, its intertwiner space is spanned by maps indexed by partitions as in Definition \ref{DefTpi}. For $\ee$-versions of easy quantum groups, the situation is a bit more delicate and we cannot give a partition approach in general. However, we may at least provide some access to the intertwiner spaces using modified partitions which we will now develop in three steps.

\subsection{Expressing relations as intertwiners}\label{SectRelInt}

Let us introduce certain linear maps extending Definition \ref{DefRCross}.

\begin{definition}
For $n\in \NN$ we define the following linear maps from $(\CC^n)^{\otimes 2}$ to $(\CC^n)^{\otimes 2}$.
\begin{align*}
R^1_{\crosspart}(e_i\otimes e_j)&:=\delta_{{\ee_{ij}=1}}e_j\otimes e_i\\
R^1_{\ididpart}(e_i\otimes e_j)&:=\delta_{{\ee_{ij}=1}}e_i\otimes e_j\\
R^0_{\ididpart}(e_i\otimes e_j)&:=\delta_{{\ee_{ij}=0}}e_i\otimes e_j\\
R^0_{\paarbaarpart}(e_i\otimes e_j)&:=\delta_{ij}\sum_k\delta_{{\ee_{ik}=0}}e_k\otimes e_k
\end{align*}
\end{definition}

We will now describe the intertwiners implementing relations such as $R^\ee$ and $\mathring R^\ee$ which we recall here (Def. \ref{DefEspOn}, \ref{DefRRing} and \ref{DefAut}):
\begin{align*}
&(R^\ee1) \quad &&u_{ik}u_{jl}=u_{jl}u_{ik}\textnormal{ if } \ee_{ij}=1 \textnormal{ and } \ee_{kl}=1\\
&(R^\ee2) \quad &&u_{ik}u_{jl}=u_{jk}u_{il}\textnormal{ if }\ee_{ij}=1\textnormal{ and }\ee_{kl}=0\\
& \quad &&\textnormal{and }u_{ik}u_{jl}=u_{il}u_{jk}\textnormal{ if }\ee_{ij}=0\textnormal{ and }\ee_{kl}=1\\
&(\mathring R^\ee2) \quad &&\delta_{\ee_{kl}=0}u_{ik}u_{jl}=\delta_{\ee_{ij}=0}u_{ik}u_{jl}\\
&(R^\ee)&&(R^\ee1)\textnormal{ and }(R^\ee2)\\
&(\mathring R^\ee)&&(R^\ee1)\textnormal{ and }(\mathring R^\ee2)\\
&(R_{\aut})\quad&&\sum_k \delta_{\ee_{kl}=1}u_{ik}=\sum_j \delta_{\ee_{ij}=1}u_{jl}
&&\qquad\qquad\qquad\qquad
\end{align*}
Moreover, we define:
\begin{align*}
&(R'^\ee2) \quad &&u_{ik}u_{jl}=\delta_{kl}\sum_m\delta_{\ee_{km}=0}u_{jm}u_{im}\textnormal{ if } \ee_{ij}=1 \textnormal{ and }\ee_{kl}=0\\
&   &&\textnormal{and }u_{ik}u_{jl}=\delta_{ij}\sum_m\delta_{\ee_{im}=0}u_{ml}u_{mk}\textnormal{ if } \ee_{ij}=0 \textnormal{ and }\ee_{kl}=1
&&\qquad\qquad\qquad\qquad\\
&(R'^\ee)&&(R^\ee1)\textnormal{ and }(R'^\ee2)
\end{align*}

\begin{lemma}\label{LemREpsIntertwiner}
Let $G\subset O_n^+$ be a compact matrix quantum group. 
\begin{itemize}
\item[(a)]  $R^1_{\crosspart}+R^0_{\ididpart}$ is an intertwiner of $G$ if and only if ($R^\ee$) holds.
\item[(b)]  $R^1_{\crosspart}+R^0_{\paarbaarpart}$ is an intertwiner of $G$ if and only if ($R'^\ee$) holds
\item[(c)]  $R^1_{\crosspart}$ is an intertwiner of $G$ if and only if ($\mathring R^\ee$) holds.
\item[(d)] $R^0_{\ididpart}$ is an intertwiner of $G$ if and only if ($\mathring R^\ee2$) holds.
\item[(e)] $\ee$ is an intertwiner of $G$ (i.e. $u\ee=\ee u$) if and only if ($R_{\aut}$) holds.
\end{itemize}
\end{lemma}
\begin{proof}
(a) We first compute:
\begin{align*}
&u^{\otimes 2}\left(R^1_{\crosspart}+R^0_{\ididpart}\right)(e_l\otimes e_k)\\
&=u^{\otimes 2}\left(\delta_{\ee_{kl}=1}e_k\otimes e_l+\delta_{\ee_{kl}=0}e_l\otimes e_k\right)\\
&=\sum_{i,j}\left(\delta_{\ee_{kl}=1}u_{ik}u_{jl}+\delta_{\ee_{kl}=0}u_{il}u_{jk}\right)\otimes e_i\otimes e_j
\end{align*}
And:
\begin{align*}
&\left(R^1_{\crosspart}+R^0_{\ididpart}\right)u^{\otimes 2}(e_l\otimes e_k)\\
&=\sum_{i,j}\left(u_{jl}u_{ik}\otimes R^1_{\crosspart}(e_j\otimes e_i)+
u_{il}u_{jk}\otimes R^0_{\ididpart}(e_i\otimes e_j)\right)\\
&=\sum_{i,j}\left(\delta_{\ee_{ij}=1}u_{jl}u_{ik}+\delta_{\ee_{ij}=0}u_{il}u_{jk}\right)\otimes e_i\otimes e_j
\end{align*}
We infer that $R^1_{\crosspart}+R^0_{\ididpart}$ is an intertwiner if and only if for all $i,j,k,l$:
\[\delta_{\ee_{kl}=1}u_{ik}u_{jl}+\delta_{\ee_{kl}=0}u_{il}u_{jk}
=\delta_{\ee_{ij}=1}u_{jl}u_{ik}+\delta_{\ee_{ij}=0}u_{il}u_{jk}\]
These relations are equivalent to $R^\ee$.

(b) We proceed like in (a) by computing:
\begin{align*}
&u^{\otimes 2}\left(R^1_{\crosspart}+R^0_{\paarbaarpart}\right)(e_l\otimes e_k)\\
&=u^{\otimes 2}\left(\delta_{\ee_{kl}=1}e_k\otimes e_l+\delta_{kl}\sum_m\delta_{{\ee_{km}=0}}e_m\otimes e_m\right)\\
&=\sum_{i,j}\left(\delta_{\ee_{kl}=1}u_{ik}u_{jl}+
\delta_{kl}\sum_m\delta_{{\ee_{km}=0}}u_{im}u_{jm}\right)\otimes e_i\otimes e_j
\end{align*}
And:
\begin{align*}
&\left(R^1_{\crosspart}+R^0_{\paarbaarpart}\right)u^{\otimes 2}(e_l\otimes e_k)\\
&=\sum_{i,j}u_{jl}u_{ik}\otimes R^1_{\crosspart}(e_j\otimes e_i)
+\sum_{p,m}u_{pl}u_{mk}\otimes R^0_{\paarbaarpart}(e_p\otimes e_m)\\
&=\sum_{i,j}\delta_{\ee_{ij}=1}u_{jl}u_{ik}\otimes e_i\otimes e_j
+\sum_{p,m,i}\delta_{pm}\delta_{\ee_{im=0}}u_{pl}u_{mk}\otimes e_i\otimes e_i\\
&=\sum_{i,j}\delta_{\ee_{ij}=1}u_{jl}u_{ik}\otimes e_i\otimes e_j
+\sum_{i,m}\delta_{\ee_{im=0}}u_{ml}u_{mk}\otimes e_i\otimes e_i\\
&=\sum_{i,j}\left(\delta_{\ee_{ij}=1}u_{jl}u_{ik}
+\delta_{ij}\sum_{m}\delta_{\ee_{im=0}}u_{ml}u_{mk}\right)\otimes e_i\otimes e_j
\end{align*}
Assertion (c) is the contents of Lemma \ref{LemRRingEpsIntertwiner}, and (d) and (e) are straightforward.
\end{proof}

We infer that any of the above relations in Lemma \ref{LemREpsIntertwiner} may be implemented by intertwiners. Hence, each of them passes through the comultiplication map $\Delta$ of compact matrix quantum groups. This means, that we may define quantum groups satisfying these relations. In this sense, the above Lemma \ref{LemREpsIntertwiner} gives a more systematic proof of Lemma \ref{LemOnQG} and  Lemma \ref{LemRZwei}. Moreover, using the building blocks ($R^\ee1$), ($\mathring R^\ee2$), ($R^\ee2$) and ($R_{\aut}$), we may define possibly new quantum subgroups of $O_n^+$ by quotienting out the following relations:
\begin{itemize}
\item only ($\mathring R^\ee2$)
\item only ($R_{\aut}$)
\item ($\mathring R^\ee2$) together with ($R_{\aut}$)
\item ($\mathring R^\ee$) together with ($R_{\aut}$)
\item ($R^\ee$) together with ($R_{\aut}$)
\end{itemize}
It is not clear, whether the relations ($R^\ee1$) and ($R^\ee2$) may be expressed separately by intertwiners, so we don't know whether we may define quantum groups by using these relations separately.

\subsection{Equivalence of relations by intertwiner calculus}\label{EqRelInt}

Having expressed our relations by intertwiners, it is very easy to deduce some relations from others or even to show their equivalence. All we need to show is that we may construct certain intertwiners from others using the operations of intertwiner spaces \cite{WoTK}:
\begin{itemize}
\item If $S$ and $T$ are intertwiners of $G$, so are $S\otimes T$, $ST$ and  $T^*$, as well as $\alpha S+\beta T$, for $\alpha,\beta\in\CC$.
\item The identity map $\id:\CC^n\to\CC^n$ is an intertwiner of $G$.
\item If $G\subset O_n^+$, then the maps $T_{\paarpart}$ and $T_{\baarpart}$ are intertwiners of $G$ (see for instance \cite[Lemma 2.5]{Web}).
\end{itemize}

Recall from Definition \ref{DefTpi} the definition of $T_\pi$, where $\pi$ is some partition.

\begin{lemma}\label{LemIntertwinerEq}
Let $G\subset O_n^+$ be a quantum subgroup of $O_n^+$. Concerning intertwiners in $G$, we have:
\begin{itemize}
\item[(a)] $R^1_{\crosspart}$ is an intertwiner if and only if $R^0_{\ididpart}$ and $R^1_{\crosspart}+R^0_{\ididpart}$ are.
\item[(b)] $R^1_{\crosspart}+R^0_{\ididpart}$ is an intertwiner  if and only if $R^1_{\crosspart}+R^0_{\paarbaarpart}$ is.
\item[(c)] $R^0_{\ididpart}$ is an intertwiner  if and only if $R^0_{\paarbaarpart}$ is.
\item[(d)] Let $T_{\vierpartrot}$ be an intertwiner of $G$. 

Then $R^1_{\crosspart}+R^0_{\ididpart}$ is an intertwiner  if and only if $R^1_{\crosspart}$ is.
\item[(e)] Let $T_{\dreipartrot}$ be an intertwiner of $G$. 

Then $R^0_{\ididpart}$ is an intertwiner if and only if $\ee$ is.
\end{itemize}
\end{lemma}
\begin{proof}
(a) We have:
\[R^0_{\ididpart}=(\id\otimes\id)-\left(R^1_{\crosspart}\right)^2\]

(b) We compute:
\begin{align*}
&\left(T_{\baarpart}\otimes\id\otimes\id\right)
\left(\id\otimes \left(R^1_{\crosspart}+R^0_{\ididpart}\right)\otimes\id\right)
\left(\id\otimes\id\otimes T_{\paarpart}\right)(e_i\otimes e_j)\\
&=\sum_k\left(T_{\baarpart}\otimes\id\otimes\id\right)
\left(\id\otimes \left(R^1_{\crosspart}+R^0_{\ididpart}\right)\otimes\id\right)
(e_i\otimes e_j\otimes e_k\otimes e_k)\\
&=\sum_k\left(T_{\baarpart}\otimes\id\otimes\id\right)
(\delta_{\ee_{kj}=1}e_i\otimes e_k\otimes e_j\otimes e_k+\delta_{\ee_{kj}=0}e_i\otimes e_j\otimes e_k\otimes e_k)\\
&=\sum_k\delta_{\ee_{kj}=1}\delta_{ik}e_j\otimes e_k+\sum_k\delta_{\ee_{kj}=0}\delta_{ij}e_k\otimes e_k\\
&=\left(R^1_{\crosspart}+R^0_{\paarbaarpart}\right)(e_i\otimes e_j)
\end{align*}
Thus, if $R^1_{\crosspart}+R^0_{\ididpart}$ is an intertwiner, so is $R^1_{\crosspart}+R^0_{\paarbaarpart}$. A similar computation shows the converse.

(c) Omitting $R^1_{\crosspart}$ in the proof of (b), we obtain the proof of (c).

(d) If $R^1_{\crosspart}+R^0_{\ididpart}$ is an intertwiner of $G$, then $R^1_{\crosspart}+R^0_{\paarbaarpart}$, too, by (b). From
\[\left(R^1_{\crosspart}+R^0_{\paarbaarpart}\right)T_{\vierpartrot}=R^0_{\paarbaarpart}\]
and  (c), we infer that $R^0_{\ididpart}$ is an intertwiner of $G$.

(e) If $R^0_{\ididpart}$ is an intertwiner of $G$, then $R^0_{\paarbaarpart}$, too, by (c). We compute:
\[T_{\dreipartrot}R^0_{\paarbaarpart}T_{\dreipartrot}^*(e_i)=\sum_k\delta_{\ee_{ik}=0}e_k\]
Thus, $\ee=\id-T_{\dreipartrot}R^0_{\paarbaarpart}T_{\dreipartrot}^*$ is an intertwiner of $G$. Conversely, if $\ee$ is an intertwiner of $G$, then also
\[R^0_{\paarbaarpart}=T_{\dreipartrot}^*(\id-\ee)T_{\dreipartrot}\]
is an intertwiner of $G$.
\end{proof}

Using the above lemma, we immediately see the consequences for the relations, recovering results from Lemma \ref{LemREpsAeqRRing} and Lemma \ref{LemRelSnEps}.
\begin{lemma}
Let $G\subset O_n^+$ be a quantum subgroup of $O_n^+$.
\begin{itemize}
\item[(a)] The relations ($\mathring R^\ee$) imply ($R^\ee$).
\item[(b)] The relations ($R^\ee$) hold if and only if the relations ($R'^\ee$) hold.
\item[(c)] If $u_{ik}u_{jk}=u_{ki}u_{kj}=0$ for all $i\neq j$, then ($\mathring R^\ee$) and ($R^\ee$) are equivalent.
\item[(d)] If $G\subset S_n^+$, then ($\mathring R^\ee2$)  and ($R_{\aut}$) are equivalent.
\end{itemize}
\end{lemma}
\begin{proof}
The relations $u_{ik}u_{jk}=u_{ki}u_{kj}=0$ for all $i\neq j$ are equivalent to the fact that $T_{\vierpartrot}$ is an intertwiner (see for instance \cite[Lemma 2.5]{Web}). Moreover, one can check that $T_{\dreipartrot}$ is an intertwiner if and only if $G\subset S_n^+$. We then use  Lemma \ref{LemIntertwinerEq} together with Lemma \ref{LemREpsIntertwiner}.
\end{proof}

\subsection{Further relations that hold in $O_n^\ee$}

\begin{proposition}\label{PropRefinedOn}
The relations $R'^{\ee}$ hold in $O_n^\ee$. In particular:
\begin{align*}
&\ee_{ij}=1,\ee_{kl}=0,k\neq l: &&u_{ik}u_{jl}=0\textnormal{ and }u_{ki}u_{lj}=0\\
&\ee_{ij}=1:\qquad && \sum_{m\neq k\textnormal{ and }\ee_{km}=0} u_{im}u_{jm}=\sum_{m= k\textnormal{ or }\ee_{km}=1} u_{im}u_{jm}=0\\
& && \sum_{m\neq k\textnormal{ and }\ee_{km}=0} u_{mi}u_{mj}=\sum_{m= k\textnormal{ or }\ee_{km}=1} u_{mi}u_{mj}=0
\end{align*}
\end{proposition}
\begin{proof}
We use Lemma \ref{LemIntertwinerEq}(b) and deduce from $R'^\ee$ and $R^\ee$ for $\ee_{ij}=1$:
\[u_{ik}u_{jk}=\sum_{m\neq k:\;\ee_{km}=0}u_{im}u_{jm}+u_{ik}u_{jk}\]
Together with $\sum_m u_{im}u_{jm}=\delta_{ij}$, this proves the claim.
\end{proof}

The above relations for sums of $u_{im}u_{jm}$ seem a bit strange at first glance. However, note that such groupings of summands are nothing unusual in the theory of quantum groups. Indeed, while we have that the sum $\sum_m u_{im}u_{jm}$ is zero for $i\neq j$ in $O_n^+$, we require that all of its summands $u_{im}u_{jm}$ are  zero in $S_n^+$ and $H_n^+$. Now, the requirement in $O_n^\ee$ is something in between: certain subsums have to be zero.

\subsection{First ideas for a partition calculus for $O_n^\ee$}\label{SectPartCalc}

In Section \ref{EqRelInt}, we saw the use of intertwiner calculus for compact matrix quantum groups. For easy quantum groups \cite{BS}, this intertwiner calculus can be transferred to a partition calculus -- the intertwiners of an easy quantum group $G$ are spanned by linear maps $T_\pi$ indexed by partitions $\pi$, and we have: If $T_\pi$ and $T_\sigma$ are intertwiners of $G$, so are $T_{\pi\otimes\sigma}$, $T_{\pi\sigma}$ and $T_{\pi^*}$. See \cite{BS} or \cite{Web} for details.
We will now develop a pictorial approach to intertwiners of $O_n^\ee$ with the help of which some of the intertwiner calculus of the preceding subsection can be done by purely pictorial means.

The relations $R^\ee$ are implemented by the intertwiner $R^1_{\crosspart}+R^0_{\ididpart}$ which is somehow a superposition of the maps $T_{\crosspart}$ and $T_{\ididpart}$ -- depending on the $\ee$-entry, we either apply $T_{\crosspart}$ or  $T_{\ididpart}$. We therefore propose the following symbolism:
\begin{align*}
S_{\MixCrossId}&:= R^1_{\crosspart}+R^0_{\ididpart}\\
S_{\MixCrossPaar}&:= R^1_{\crosspart}+R^0_{\paarbaarpart}\\
S_{\MixIdPaar}&:= R^1_{\ididpart}+R^0_{\paarbaarpart}
\end{align*}
Note that $R^1_{\ididpart}+R^0_{\ididpart}=\id\otimes\id$.
With this graphical calculus, the proof of Lemma \ref{LemIntertwinerEq}(b) reads as follows, following the black lines inside the box for $\ee_{ij}=1$ and the gray ones for $\ee_{ij}=0$:

\setlength{\unitlength}{0.5cm}
\begin{center}
\begin{picture}(8,12)
\put(0,5){\MixCrossPaar}
\put(2,5.2){$=$}
\put(2,5){\partii{1}{1}{2}}
\put(2,5){\parti{3}{3}}
\put(2,5){\parti{3}{4}}
\put(2,7.5){\parti{2}{1}}
\put(2,7.5){\parti{2}{4}}
\put(3.9,5){\MixCrossId}
\put(2,7){\upparti{3}{1}}
\put(2,7){\upparti{3}{2}}
\put(2,7){\uppartii{1}{3}{4}}
\end{picture}
\end{center}

Moreover, we easily see on the graphical level (the composition of linear maps meaning that we put one picture above the other and follow the lines):
\begin{align*}
S_{\MixCrossId}S_{\MixCrossId}&=T_{\ididpart}\\
S_{\MixCrossId}S_{\MixCrossPaar}&=S_{\MixCrossPaar}S_{\MixCrossId}=S_{\MixIdPaar}
\end{align*}
But note that:
\[S_{\MixCrossPaar}S_{\MixCrossPaar}(e_i\otimes e_j)=R^1_{\ididpart}(e_i\otimes e_j)+\delta_{ij}\sum_k\left(\sum_l\delta_{\ee_{il}=0}\delta_{\ee_{lk}=0}\right)e_k\otimes e_k\]
Thus:
\[S_{\MixCrossPaar}S_{\MixCrossPaar}\neq S_{\MixIdPaar}\]
This can also be seen on the graphical level: The composition of the gray lines yields a loop. Therefore, we need to extend our pictures by keeping track of the number of gray loops, and our pictorial representation explodes -- just like the intertwiner space itself. This seems to put an end to the attempt of describing the intertwiner spaces of $O_n^\ee$ graphically.  However, depending on $\ee$, the number of loops needed in an extended definition of $S_{\MixCrossPaar}$ might be limited. Indeed, if $\ee$ is as in Example \ref{ExEps}(f), then:
\[\sum_l\delta_{\ee_{il}=0}\delta_{\ee_{lk}=0}=\delta_{\ee_{ik}=0}+\delta_{ik}+1\]
Therefore $S_{\MixCrossPaar}S_{\MixCrossPaar}$ is again in the span of the already known intertwiners, since:
\[S_{\MixCrossPaar}S_{\MixCrossPaar}=S_{\MixIdPaar}+T_{\vierpartrot}+T_{\paarbaarpart}\]

\section{Open questions}\label{OpProbl}

We want to finish this articles with a list of open questions for future work on $\ee$-mixed quantum groups.

\subsection{About $\mathring O_n^\ee$ and the choice $\ee_{ii}=0$.}

Recall that $\mathring S_n^\ee=S_n^\ee$ and $\mathring H_n^\ee=H_n^\ee$, while $\mathring O_n^\ee\neq O_n^\ee$ and $\mathring B_n^\ee\neq B_n^\ee$ in general. We therefore ask: For which matrices $\ee$ do we have $\mathring O_n^\ee\neq O_n^\ee$? Recall from Section \ref{SectIntertwinersSn} that  $\mathring O_n^\ee= O_n^\ee=O_n^+$ for $\ee=\ee_{\free}$, but $\mathring O_n^\ee=H_n\neq O_n=O_n^\ee$  for $\ee=\ee_{\comm}$. This is due to some subtlety with respect to the choice of the diagonal entries of $\ee$: The decision to put $\ee_{ii}=0$ or $\ee_{ii}=1$ does not affect the definition of $\ee$-independence, nor the relations $R^\ee$ (note that the intertwiner $R_{\crosspart}^1+R_{\ididpart}^0$ does not depend on the choice of the diagonal entries). However, the choice of the cumulants in \cite{SpW} is affected by the choice of the diagonal entries, and putting them to zero is the choice for the free cumulants rather than for the classical ones. Thus, in some cases, our choice of the diagonal entries fits better with the free situation -- as may be seen with the relations $\mathring R^\ee$ and the discussion on $\mathring O_n^\ee$. Note that the intertwiner $R_{\crosspart}^1$, which implements the relations $\mathring R^\ee$, \emph{does} depend on the choice of the diagonal entries.

We conclude that depending on the choice of the diagonal entries, we have yet another possible definition of the relations $\mathring R^\ee$, and yet another possible definition of $\mathring O_n^\ee$.

Note that in case $\mathring O_n^\ee\neq O_n^\ee$, the quantum group $\mathring O_n^\ee$ acts on the $\ee$-sphere $S^{n-1}_{\RR,\ee}$, but the action is not maximal. On which quantum space does it act maximally? We don't know. It would be interesting to analyze the spheres with the relations $x_ix_j=0$ for $\ee_{ij}=0$, either combined with $x_ix_j=x_jx_i$ for $\ee_{ij}=1$ or separately.

Furthermore, Lemma \ref{LemRZwei} allows us to define yet another variant of $O_n^\ee$ by taking the quotient of $O_n^+$ by the relations ($\mathring R^\ee2$) of Definition \ref{DefRRing}. This time, we obtain a quantum group which contains $O_n^\ee$ as a subgroup, and we might want to study this quantum group, too. See also the end of Section \ref{SectRelInt} for further possible definitions of quantum groups.

\subsection{About $S_n^\ee$}

We know from Proposition \ref{OnEpsNoncommutative}, that $O_n^\ee$ is always noncommutative. However, we have seen in Examples \ref{ExKey} and \ref{ExConverse}, that this is not always the case for $S_n^\ee$. Can we characterize the noncommutativity of $S_n^\ee$ in terms of properties of $\ee$? This question is deeply linked with the investigations of graphs having no quantum symmetry, as treated in \cite{BanBicQAG} and \cite{QAGThesis}. Concerning the link to Banica's quantum automorphism groups $S_n^{\Gamma_\ee}$ of graphs $\Gamma_\ee$, it would be interesting to find examples such that $T_n^\ee\neq S_n^\ee\neq S_n^{\Gamma_\ee}$. Moreover, do we have $S_n^\ee\neq S_n^{\ee'}$ for $\ee\neq \ee'$ in analogy to the case of $O_n^\ee$ (see Proposition \ref{OnEpsNoncommutative})? 

Also, it would be interesting to know whether or not $C(S_5^\ee)$ is commutative for $\ee$ as in Example \ref{ExEps}(f), since this matrix cannot be obtained from iterated grouping of independences. Likewise we would like to know whether $S_6^\ee$ is nontrivial for $\ee$ as in Example \ref{ExTn}(c). Here, $T_6^\ee$ is the trivial group, so the question is, whether we may find some nontrivial quantum group $S_n^\ee$ such that $T_n^\ee$ is trivial. Finally, in \cite{BanBicFour}, Banica and Bichon classify all quantum subgroups of $S_4^+$. It would be interesting to know which of them are of the type $S_n^\ee$ as studied in our article, and which are not.

\subsection{About $T_n^\ee$}

Can we read off the group $T_n^\ee\subset S_n$ of Definition \ref{DefTn} from the matrix $\ee$? Can we link it to the Coxeter group $\ZZ_2^\ee$? How does $T_n^\ee$ look like for iterated grouping of variables as in Example \ref{ExEps}? What is $T_n^\ee$ for $\ee$ as in Example \ref{ExEps}(f)?

Furthermore,  it would be interesting to see whether $T_n^\ee$ is a kind of an invariant for $O_n^\ee$ or $S_n^\ee$. Note that the statement $O_n^\ee\neq O_n^{\ee'}$ (or $S_n^\ee\neq S_n^{\ee'}$) is a rather weak one since this  only means that there is no $^*$-isomorphism between the associated $C^*$-algebras sending generators to generators. What about other isomorphisms? Can they be excluded with the help of $T_n^\ee$?

\subsection{Extensions of the de Finetti theorem}

Recapturing the proof of Theorem \ref{deFinetti}, we observe that the crucial point is Proposition \ref{CorIntertwiner}: Is $\delta_{\pi\in NC_{\mathcal C}^\ee[i]}$ an intertwiner of $G$? The input data for proving it is:
\begin{itemize}
\item[(i)] $T_\pi$ is an intertwiner of $G$, for $\pi\in\mathcal C$ a noncrossing partition.
\item[(ii)] $R_{\crosspart}^1$ is an intertwiner of $G$.
\end{itemize}
Now, for quantum groups $\mathring G_n^\ee$, item (ii) is true, by Lemma \ref{LemRRingEpsIntertwiner}. Thus, from a de Finetti point of view, $\mathring O_n^\ee$ is more natural than $O_n^\ee$. On the other hand, $O_n^\ee$ is more natural from a quantum action point of view (Theorem \ref{Action}). Hence, it would be interesting to find de Finetti theorems also for $R^\ee$, in particular for $O_n^\ee$ and $B_n^\ee$.

\subsection{A partition calculus for $\ee$-quantum groups}

As sketched in Section \ref{SectPartCalc}, a general partition calculus for $O_n^\ee$ seems hopeless, since the graphical calculus seems to explode. This fits (in the easy quantum group philosophy) with the fact that the quantum group containing the maximal intertwiner space, namely $T_n^\ee$, can be trivial (which means that its intertwiner spaces consists of all possible linear maps). however, it would be interesting to study the intertwiner spaces in special examples of $\ee$. Note that for $\ee=\ee_{\comm}$ and $\ee=\ee_{\free}$ we do have such a partition calculus, since $S_{\MixCrossPaar}=T_{\crosspart}$ in the former case and $S_{\MixCrossPaar}=T_{\ididpart}$ in the latter. Moreover, Example \ref{ExEps}(f) seems to be promising for providing an example with an accessible intertwiner space structure.

However, we need to define pictures also for more general intertwiners. The idea is to take any partition $\pi\in P(k,l)$ and to replace each crossing by one of the basic boxes such as:
\[\MixCrossPaar\;,\MixCrossId\;,\MixIdPaar\;,\textnormal{loops}\ldots\]
In order to have a partition apporach to the intertwiners of quantum subgroups of $O_n^\ee$, for instance for $S_n^\ee$, the situation is even more involved: We need to come up with symbols for partitions with blocks of arbitrary sizes -- and it is not even clear how to define crossings for such partitions.

In any case, any progress with respect to the decription of the intertwiner spaces of $O_n^\ee$ or $S_n^\ee$ will certainly extend our understanding of quantum subgroups of $O_n^+$.

\section*{Acknowledgements}

We thank Guillaume C\'ebron for discussions on $\ee$-independence.

\bibliographystyle{alpha}
\nocite{*}
\bibliography{BibFileMixedQG}

\end{document}